\documentclass[11pt,reqno]{amsart}
\usepackage{amsmath,amsthm,amsfonts}

\usepackage[utf8]{inputenc}
\usepackage{fullpage}
\usepackage{bbm}
\usepackage{bm}
\usepackage{amssymb}
\usepackage{shuffle}
\usepackage{dirtytalk}
\usepackage{enumerate}
\usepackage{enumitem}
\usepackage{xcolor}
\usepackage{mathtools}
\usepackage{amsmath}

\usepackage{amsxtra,setspace,xspace,lmodern,psfrag,color,latexsym}
\usepackage{pifont,mathrsfs,caption,microtype,accents}
\usepackage[bookmarks=true, colorlinks=true]{hyperref}
\hypersetup{urlcolor=blue,citecolor=red,linkcolor=black}
\usepackage{cancel}
\usepackage[normalem]{ulem}

\newcommand{\M}[0]{\mathcal{M}}
\newcommand{\N}[0]{\mathbb{N}}
\newcommand{\onabla}[0]{\overline{\nabla}}

\newcommand{\V}[0]{\mathcal{V}}
\newcommand{\R}[0]{\mathbb{R}}

\renewcommand{\S}[0]{\mathcal{S}}

\newcommand{\bj}{\bm{j}}

\newcommand{\w}[0]{\omega}
\newcommand{\norm}[1]{\lVert #1 \rVert}

\newcommand{\dd}{\mathrm{d}}
\newcommand{\Rddiag}{{\R^{2d}_{\!\scriptscriptstyle\diagup}}}
\newcommand{\cMtv}{\mathcal{\M}_{\mathrm{TV}}}
\DeclareMathOperator{\TV}{{TV}}
\DeclareMathOperator{\AC}{AC}
\newcommand{\ACt}{\mathcal{AC}_T}

\DeclareMathOperator{\NCET}{NCE_\mathrm{T}}
\DeclareMathOperator*{\esssup}{ess\,sup}
\DeclareMathOperator*{\supp}{supp}
\newcommand{\Rd}[0]{\R^d}

\newtheorem{definition}{Definition}[section]
\newtheorem{proposition}{Proposition}[section]
\newtheorem{lemma}{Lemma}[section]
\newtheorem{theorem}{Theorem}[section]
\newtheorem{remark}{Remark}[section]
\newtheorem{example}{Example}[section]

\newenvironment{listi}
{\begin{list}
    {(\roman{broj})}
    { \usecounter{broj}}
    \setlength{\labelwidth}{30pt}
    \addtolength{\itemsep}{2pt}}
  {   \end{list} }
\newcounter{broj}

\numberwithin{equation}{section}

\makeatletter
\@namedef{subjclassname@2020}{\textup{2020} Mathematics Subject Classification}
\makeatother

\title{On evolution PDEs on co-evolving graphs}
\author{Antonio Esposito \and L\'aszl\'o Mikol\'as}
\address{Mathematical Institute, University of Oxford, Woodstock Road, Oxford, OX2 6GG, United Kingdom.}
\email{antonio.esposito@maths.ox.ac.uk}
\email{laszlo.mikolas@maths.ox.ac.uk}

\begin{document}

\begin{abstract}
We provide a well-posedness theory for a class of nonlocal continuity equations on co-evolving graphs. We describe the connection among vertices through an edge weight function and we let it evolve in time, coupling its dynamics with the dynamics on the graph. This is relevant in applications to opinion dynamics and transportation networks. Existence and uniqueness of suitably defined solutions is obtained by exploiting the Banach fixed-point theorem. We consider different time scales for the evolution of the weight function: faster and slower than the flow defined on the graph. The former leads to graphs whose weight functions depend nonlocally on the density configuration at the vertices, while the latter induces static graphs. Furthermore, we prove a discrete-to-continuum limit for the PDEs under study as the number of vertices converges to infinity.
\end{abstract}






\maketitle

\section*{Notation} We list notation we shall use throughout the manuscript for reference. 

\begin{itemize}
    \item $A$ denotes a generic subset of $\Rd$.
    \item $\mathcal{B}(A)$: Borel subsets of $A$.
    \item  $\mathcal{M}(A)$: Radon measures on $A$.
    \item $\mathcal{M}^{+}(A)$: nonnegative Radon measures on $A$.
    \item Given $\nu \in \mathcal{M}(\mathbb{R}^d)$ and letting $A \in \mathcal{B}(\mathbb{R}^d)$, we denote by $\nu^{+}(A):=\sup _{B \in \mathcal{B}(A)} \nu(B)$ and $\nu^{-}(A):=-\inf _{B \in \mathcal{B}(A)} \nu(B)$ the upper and lower variation measures of $\nu$; the total variation measure of $\nu$ is $|\nu|(A):=\nu^{+}(A)+\nu^{-}(A)$ and its total variation norm is $\|\nu\|_{\mathrm{TV}}:=|\nu|(\mathbb{R}^d)$.
    \item $\mathcal{M}_{\mathrm{TV}}(A)$ : Radon measures on $A$ with finite total variation.
    \item $\mathcal{M}_{\mathrm{TV}}^{+}(A):=\mathcal{M}^{+}(A) \cap \mathcal{M}_{\mathrm{TV}}(A)$.
    \item $C_0(\Rd)$ is the space of continuous functions on $\Rd$ vanishing at infinity. 
    \item $C_b(\Rd)$ is the space of continuous and bounded functions on $\R^d$.
    \item  $\Rddiag:=\left\{(x, y) \in \mathbb{R}^d \times \mathbb{R}^d: x \neq y\right\}$ is the off-diagonal of $\mathbb{R}^d \times \mathbb{R}^d$.
\item $\mu\in\mathcal{M}^{+}(\mathbb{R}^d)$ is referred to as the base measure and acts as an abstract notion of vertices.
\item $\eta:\Rddiag\to\R$ is the edge weight function.
\item $G$ is the set of edges; i.e., $G=\left\{(x, y) \in \Rddiag: \eta(x, y)\neq 0\right\}$.
\item $\mathcal{V}^{\mathrm{as}}(\Rddiag)$ is the set of antisymmetric vector fields on $\Rddiag$; that is, $\mathcal{V}^{\mathrm{as}}(\Rddiag)=\{v: \Rddiag \rightarrow$ $\left.\mathbb{R}: v(x,y)=-v(y,x)\right\}$.
\item ${\onabla} \phi(x, y)=\phi(y)-\phi(x)$ denotes the nonlocal gradient for $\phi:\R^d\to\R$.
\item $\onabla\cdot j$ is a nonlocal divergence for a flux $j\in\mathcal{M}(\Rddiag)$, cf.~Definition~\ref{def:flux}.
    \item $T$ is a positive and finite final time.
    \item $AC([0,T]; \M_{TV}(\R^d))$ is the space of absolutely continuous curves with respect to $\norm{\cdot}_{TV}$ from $[0,T]$ to $\M_{TV}(\R^d)$.
    \item Given $a \in \R, a_+ := \max \{0,a\}$ and $a_-:=(-a)_{+}$ are its positive and negative parts, respectively.
    \item We shall denote the time dependence by using subscripts. For instance, for a curve $\rho\in AC([0,T]; \M_{TV}(\R^d))$ we use $\rho_t$ to denote an element in $\cMtv(\Rd)$, for any $t\in[0,T]$.
\end{itemize}

\section{Introduction}

In this manuscript we study a class of partial differential equations (PDEs) on graphs whose underlying structure is itself evolving — we shall refer to it as \textit{co-evolving graphs}. More precisely, in contrast with \textit{static graphs}, we allow the link between vertices to change in time, depending on the dynamics on the graph, according to another equation. Our work is motivated by the recent interest in the study of evolution equations on graphs and networks, due to possible applications in several real-world phenomena where individuals interact if they are interconnected in specific ways. In social networks, for example, one can model the spread of opinions, or behaviours, by assigning probabilities for individuals to adopt certain attitudes based on their neighbours' choices. This is useful to model polarisation and formation of echo chambers, cf.~for example~\cite{Baumann_al_PRL_20}. Another possible application concerns transportation networks, where the flux from one vertex to a connected one depends on some scalar quantities at the neighbouring vertices. This process is known as generalised gravity interaction and it gives rise to diffusion-localisation models on networks with nontrivial dynamics, see~\cite{Koike_JSP22} and the references therein. We also mention social norm formation, \cite{Kohne2020}, and biological transport networks,~\cite{Albi2017}. The analysis of dynamic models on co-evolving or adaptive graphs has received an increasing amount of attention in recent years, as it provides a more realistic modelling tool compared to static or dynamic, but non-coupled, network models. Adaptive network models have been used to study a wide range of problems from neuroscience to game theory and economics. We refer the reader to the survey paper~\cite{berner2023adaptive} for an introduction to the field of adaptive networks and their further potential applications.

In this work, we consider a class of nonlocal continuity equations on co-evolving graphs extending the results in~\cite{esposito2021nonlocal,esposito2022class}, where the underlying graph is, instead, static. More precisely, in~\cite{esposito2021nonlocal,esposito2022class} the authors consider equations on the time interval $[0,T]$ of the form
\begin{subequations}
\label{eq:intro-CE-PDE}
\begin{align}\label{eq:intro-CE-PDE-1}
    \partial_t\rho_t + \onabla\cdot \bj_t &= 0 ,
\end{align}
with nonlocal divergence (cf.~Defintion~\ref{def:non_local_grad_div})
$$\onabla\cdot \bj_t(\dd x)=\int_{\R^d\setminus\{x\}}\eta(x,y)\dd \bj_t(x,y),$$
for a time-dependent measure flux, $\bj_t\in\mathcal{M}(G)$, on the set of edges defined by an edge weight function $\eta:\R^d\times\R^d\setminus\{x=y\}\to[0,\infty)$, i.e. $G=\{(x,y)\in\R^d\times\R^d: x\neq y \mbox{ and } \eta(x,y)>0\}$; throughout the article we set $\Rddiag:=\Rd\times\Rd\setminus\{x=y\}$. In particular, points in $\R^d$ are possible vertices and the edges, $G$, are defined through the function $\eta$. Equation~\eqref{eq:intro-CE-PDE-1} describes the time-evolution of a probability measure, $\rho_t$, representing the mass at a vertex $x\in\R^d$, according to a nonlocal continuity equation on the graph. The variation of the mass on each vertex is given as an average of all possible outgoing and ingoing fluxes. We note this is a substantial difference with equations on $\Rd$, where one usually describes the flow of moving particles with a given mass. On graphs, particles are vertices which are fixed, in contrast to the mass which is transported along the edges. In the graph setting, fluxes and velocities, $v$, are defined on the edges, whereas the mass is a vertex-based quantity. Therefore, one needs to consider a suitable interpolation of the mass at the vertices in order to have an edge-based quantity for the mass as well. Hence, on graphs, the relation between flux and velocity can be nonlinear as it depends on the interpolation function chosen, denoted by $\Phi$, as well as on the vertices, which are represented by a measure $\mu\in\mathcal{M}^+(\R^d)$ --- this choice allows to consider finite and infinite graphs in a unified framework: a finite graph can be obtained by choosing $\mu=\mu^n=\sum_{i}^n\delta_{x_i}/n$, for $\{x_1,\dots,x_n\}\subset\R^d$. In view of these considerations,~\eqref{eq:intro-CE-PDE-1} is complemented with the constitutive equation
\begin{align}
\bj=F^\Phi[\mu;\rho,v],\label{eq:intro-CE-PDE-2}
\end{align}
\end{subequations}
for admissible interpolation functions specified in Definition~\ref{def:admissible_interpolation} and Example~\ref{ex:example_interpolations}. In particular, admissible fluxes $F^\Phi$ (see Definition \ref{def:flux}) depend on the interpolation function chosen as in~\eqref{eq:intro-CE-PDE-2}.

In this manuscript, we consider a scenario where the connection of the graph evolves in time, possibly depending on the dynamics of the mass on the edges. Specifically, for $\eta:[0,T]\times (\R^d\times\R^d\setminus\{x=y\})\to \R$, we shall study the following class of systems
\begin{equation}\label{eq:general_co_evolving_NCE}
\begin{split}
&\partial_t \rho_t + \onabla \cdot F^\Phi[\mu, \eta_t ; \rho_t,V_t[\rho_t]]=0, \\
&\partial_t \eta_t (x,y) = H(t,x,y,\eta_t,\rho_t),
\end{split}
\end{equation}
for an admissible flux $F^\Phi$ satisfying Definition~\ref{def:flux}. The velocity field is given as $V:[0,T]\times\cMtv(\R^d)\times\Rddiag \to \V^{\mathrm{as}}(\Rddiag)$ (set of antisymmetric velocity fields), including possible dependence on the mass $\rho$. Among others, a driving example is given by $V_t[\rho_t](x,y)=-\onabla (K*\rho_t)(x,y)$, for $K:\R^d\times\R^d\to\R$ being an interaction potential, studied in~\cite{esposito2021nonlocal,esposito2023graphtolocal,EHPS23}. We shall focus on
\[
H(t,x,y,\eta_t,\rho_t):=\w_t[\rho_t](x,y)-\eta_t(x,y),
\]
where $\w:[0,T]\times \cMtv(\R^d)\times \Rddiag  \to  \R$ is a function satisfying a set of assumptions specified in Section~\ref{sec:well_posed}. With this choice for the function $H$ the evolution of the graph's weights is influenced by two processes. First, the connection between two vertices can depend (in a nonlocal way) on the mass configuration on the graph, as well as on edges $(x,y) \in G$ and time. A possible example for the function $\w$ could be a nonlocal functional on $\cMtv(\R^d)$, depending on the edge considered as well as on time, such as
\begin{equation}\label{eq:example_omega_intro}
\w_t[\sigma](x,y) = \int_{\R^d}K(t,x,y,z)\dd\sigma(z),
\end{equation}
where $K\in C_b([0,T]\times \Rddiag\times\R^d)$ is a general convolution kernel. The second term in the definition $H$ can be thought of as a relaxation effect in time, since, in case the first term is constant, the weight function would converge to a steady state. A similar choice is also considered, e.g., in the recent work~\cite{Burger_kinetic_net22} in a different framework of co-evolving networks.
Under these assumptions, system~\eqref{eq:general_co_evolving_NCE} reads
\begin{equation}\label{eq:ivp}
   \begin{split} 
    \partial_t \rho_t & = - \onabla \cdot F^\Phi[\mu, \eta_t ; \rho_t,V_t[\rho_t]],  
    \\
    \partial_t \eta_t & =  \w_t[\rho_t] -\eta_t, 
    \end{split}\tag{Co-NCL}
\end{equation}
for given initial data $\rho^0 \in \cMtv^M(\R^d)$ and $\eta^0\in C_{b}(\Rddiag)$. The acronym NCL stands for nonlocal conservation law as we allow the velocity field to depend on the solution itself. We add the prefix ``Co'', standing for \say{co-evolving}, to make clear we consider coupled dynamics for the weight function. Differently from previous contributions in the literature, we allow the weight function to become negative and not be symmetric. We analyse different time scales and note that these can lead to weight functions depending on the mass configuration at the vertices, which have not been explored in depth so far, to the best of our knowledge. More precisely, when the graph evolves faster than the mass on the vertices, one obtains $\eta_t=\w_t[\rho]$, leading to a possible further nonlocality in the flux provided by $\eta$, in case of $\omega$ as in~\eqref{eq:example_omega_intro}. In particular, the PDE takes the form
\[
\partial_t \rho_t = - \onabla \cdot F^\Phi[\mu_t, \w_t[\rho_t] ; \rho_t,V_t[\rho_t]].
\]
In this scenario we still talk about \textit{co-evolving} graphs, though edge weights only change according to the mass configuration and not a coupled dynamics. Furthermore, we also provide a discrete-to-continuum limit for solutions of~\eqref{eq:ivp} when the number of vertices tends to infinity. In particular, we rigorously justify the continuum equation on the graph as the many vertices limit of the discrete, finite vertices, model. This is usually referred to as the \textit{graph} or \textit{continuum} limit.

\medskip

Several authors have considered interacting particle systems on graphs to allow for heterogeneous interactions among agents, focusing on the rigorous derivation of the \textit{graph} or \textit{continuum} limit, i.e. finding the limiting dynamics as the number of particles (vertices) goes to infinity. Besides its mathematical relevance, this problem is important in many fields of application such as the study of opinion formation in models of opinion dynamics or transport networks in biology. Among others, we mention~\cite{Ayi_Pouradier_JDE21} where a model of opinion dynamics is considered and the \textit{influence} of a given agent evolves in time depending on the opinion of the others. Related results in this direction are~\cite{paul2022microscopic,porat2023mean, ayi2023graph}. Kinetic equations on co-evolving networks are considered in~\cite{Burger_kinetic_net22}, where the large population limit is studied via Liouville-type equations for the joint measure of the particles as well as the weights of the graph connecting them. Opinion formation on evolving networks is also studied in~\cite{fagioli2023opinion}. We mention the work~\cite{Gkogkas_Kuehn_Xu_CMS}, where the authors study a Kuramoto-type model in a weighted network, whose weights are allowed to depend on the phase of the oscillators. In the previous works the graph is a way of keeping track of the identity or label of different particles and does not bring further structural properties. More precisely, the graph does not determine the state space where particles evolve, but only affects the nature of the interaction among them. From this point of view, particles can be considered as point-masses, i.e. the mass is fixed, but they are allowed to move in space.

In this paper, we consider a different problem in which a weighted graph determines the ambient space, that is, positions in space are fixed, while the mass on the vertices evolves in time along the edges, whose weight can evolve as well. This distinction becomes particularly relevant in applications in data science and machine learning. For example, the popular mean-shift algorithm for clustering tasks can be understood in the framework of a continuity equation on a graph~\cite{craig2021clustering}. In a nutshell, this method attempts to find clusters in the data depending on the density of the distribution of the point cloud through different regions in space. In this setting, positions remain fixed as they represent the position of a given point in a data cloud and this allows to ensure that the mode of the density discovered by the mean-shift algorithm is indeed a data point. We refer the reader to~\cite{LauxLelmi2022+}, for instance, for further graph-based clustering algorithms, with a nonlocal dynamic.

We conclude the introduction reviewing related results in the literature in the static scenario. In~\cite{CHLZ12},~\cite{Maas11}, and~\cite{Mielke2011gradient}, the authors introduced independently the concept of Wasserstein metrics on finite graphs. The nonlinear heat equation on graphs was studied in~\cite{Medv14_SIAM,Medv19}, while a well-posedness theory for the generalised porous medium equation on infinite graphs can be found in~\cite{Bianchi_PME_Graphs_CVPDE}. In~\cite{esposito2021nonlocal}, the focus is on nonlocal dynamics on graphs using an upwind interpolation, while~\cite{esposito2022class} concerns a class of continuity equations on graphs with general interpolations. The analysis of~\cite{esposito2021nonlocal} is extended to two species with cross-interactions as well as nonlinear mobilities and $\alpha$-homogeneous flux-velocity relations for $\alpha > 0$ in~\cite{HPS21,HeinzePiSch-2}. Graphs give rise to interesting discrete-to-continuum problems, such as those considered in~\cite{giga2022graph}, for the total variation flow and the Allen--Cahn flow. Recently, \cite{esposito2023graphtolocal,EHPS23} study graphs as nonlocal approximation of nonlocal interaction equations on graphs. This is linked to discrete-to-continuum evolution problems using tessellations,~\cite{DisserLiero2015,forkertEvolutionaryGammaConvergence2020,hraivoronska2023diffusive,HraivoronskaSchlichtingTse2023}. Structures resembling graphs have been also noticed in collision dynamics in kinetic theory, see~\cite{EspositoGvalaniSchlichtingSchmidtchen2021}. The authors introduce a \textit{nonlocal collision metric} in order to propose a kinetic interpretation of the nonlocal aggregation equation. The underlying state space resembles a graph and the PDE under study takes the form of a local-nonlocal continuity equation. Metric and asymptotic properties of nonlocal Wasserstein distances are considered in~\cite{SlepcevWarren2022}. An interesting property of graphs is that they potentially represent alternative space-discretisations, resembling tessellations for finite volume schemes. Therefore, we point out a natural connection to numerical schemes for gradient flows in the Wasserstein space as in, e.g.,~\cite{BailoCarrilloHu2018,Cances2020,Bailo_etal2020,HraivoronskaSchlichtingTse2023}. We conclude the literature review by mentioning the work~\cite{Falco_PRE22}, 
where a new perspective on the link between random walks on networks and diffusion PDEs is provided. 

\subsection*{Structure of the manuscript} We introduce the set up of the problem in Section~\ref{sec:setup}, including preliminary results. In Section~\ref{sec:well_posed}, we prove well-posedness for~\ref{eq:ivp} by means of a fixed-point argument. Section~\ref{sec:slow_fast_versions} is devoted to the analysis of different time scales for~\eqref{eq:ivp} distinguishing between the graph evolving at a faster or slower rate than the dynamics of the weight function. We conclude the manuscript with a discrete-to-continuum limit for~\eqref{eq:ivp} in Section~\ref{sec:discrete_to_continuum}.

\section{Preliminaries}\label{sec:setup}
A graph is identified by a pair $(\mu,\eta)$: $\mu \in \M^+(\R^d)$ is referred to as \textit{base measure} and $\eta:\Rddiag \to \R$ is the \textit{weight function}, the analogue of the weighted edges for a discrete finite graph. We define the set of edges as $G:=\{(x,y) \in \Rddiag: {\eta(x,y) \neq 0}\}$. More precisely, the pair $(\mu, \eta)$ defines a weighted graph. In previous works, the weight function $\eta$ is sometimes non-negative, whilst in this manuscript we allow for possible negative values. Furthermore, we do not restrict to symmetric weights.

The mass configuration at the vertices is described by a measure with finite total variation, i.e. $\rho\in\cMtv(\R^d)$, as defined in the notation list. We equip the set $\cMtv(\R^d)$ with the total variation norm:
\[
\norm{\sigma}_{TV}=|\sigma|[\Rd]= \sup\bigg\{\langle \varphi, \sigma\rangle \ : \ \varphi \in C_0(\Rd), \ \norm{\varphi}_\infty\leq 1\bigg\} \ ,
\]
for any $\sigma \in \M_{TV}(\Rd)$, where $\langle \varphi, \sigma \rangle:=\int_{\Rd}\varphi\,\dd\sigma $ the dual product between $C_0(\R^d)$ and $\mathcal{M}(\R^d)$. The dynamics we consider preserve the mass, as well as the bound on the total variation, therefore, we shall consider the set
\[
\cMtv^M(\R^d):=\left\{\rho\in\cMtv(\R^d):|\rho|(\R^d)\leq M\right\}.
\]

As we are not in the traditional Euclidean setting, we recall the notion of gradient and divergence for a function defined on graphs, similar to~\cite{esposito2022class}; the only difference is that we do not restrict fluxes on the set $G$, as this is changing in time, but rather consider their (possible) extension to $\Rddiag$, cf.~Remark~\ref{rem:flux_eta_dependent}.
\begin{definition}[Nonlocal gradient and divergence]\label{def:non_local_grad_div}
For any $\phi: \mathbb{R}^d \rightarrow \mathbb{R}$, we define its nonlocal gradient ${\onabla} \phi: \Rddiag \rightarrow \mathbb{R}$ by
$$
{\onabla} \phi(x, y)=\phi(y)-\phi(x), \quad \text { for all }(x, y) \in \Rddiag .
$$
For any Radon measure $\boldsymbol{j} \in \mathcal{M}(\Rddiag)$, its nonlocal divergence ${\onabla} \cdot \boldsymbol{j} \in \mathcal{M}(\mathbb{R}^d)$ is defined as the adjoint of ${\onabla}$, i.e., for any $\phi: \mathbb{R}^d \rightarrow \mathbb{R}$ in $C_0(\R^d)$, there holds
$$
\begin{aligned}
\int_{\mathbb{R}^d} \phi \mathrm{d} {\onabla} \cdot \boldsymbol{j} & =-\frac{1}{2} \iint_\Rddiag {\onabla} \phi(x, y)\mathrm{d} \boldsymbol{j}(x, y) \\
& =\frac{1}{2} \int_{\mathbb{R}^d} \phi(x) \int_{\mathbb{R}^d \backslash\{x\}} (\mathrm{d} \boldsymbol{j}(x, y)-\mathrm{d} \boldsymbol{j}(y, x)) .
\end{aligned}
$$
In particular, for $\boldsymbol{j}$ antisymmetric, that is, $\boldsymbol{j} \in \mathcal{M}(\Rddiag)$ and $\dd\boldsymbol{j}(x,y)=-\dd\boldsymbol{j}(y,x)$, denoted $\boldsymbol{j} \in$ $\mathcal{M}^{\mathrm{as}}(\Rddiag)$, we have
$$
\int_{\mathbb{R}^d} \phi \mathrm{d} {\onabla} \cdot \boldsymbol{j}=\iint_\Rddiag \phi(x)\mathrm{d} \boldsymbol{j}(x, y) \ .
$$
\end{definition}
Solutions of the nonlocal continuity equations are intended as curves defined on a time interval $[0,T]$, for $T>0$. Therefore, we shall denote by $\AC([0,T];\cMtv(\R^d))$ the set of curves from $[0,T]$ to $\cMtv(\R^d)$ absolutely continuous with respect to the $\TV$ norm, that is, the set of curves $\rho\colon [0,T] \to \cMtv(\R^d)$
such that there exists $m\in L^1([0,T])$ with
\begin{equation*}
  \norm{\rho_t-\rho_s}_{\TV} \leq \int_s^t m(r) \dd r, \qquad \text{for all } 0\leq s< t\leq T. 
\end{equation*}
With these preliminary notions we specify the definition of weak solution for the nonlocal continuity equation, referred to as \eqref{eq:nce-expr}. 
\begin{definition}[Weak solution for the NCE]
  \label{def:nce-flux-form}
  A measurable pair of curves $(\rho,\bj)\colon [0,T] \to \cMtv(\R^d)\times \M(\Rddiag)$ is a \textit{weak} solution to the nonlocal continuity equation, denoted as 
  \begin{equation}\tag{NCE}
    \label{eq:nce-expr}
    \partial_t \rho + \onabla \cdot\bj = 0,
  \end{equation}
  provided that, for any $\varphi \in C_0(\R^d)$, it holds:
  \begin{listi}
    \item $(t\mapsto \rho_t) \in \AC([0,T];\cMtv(\R^d))$ ;
    \item $(\bj_t)_{t\in[0,T]}$ is Borel measurable and the map $(t\mapsto \langle \varphi,\onabla\cdot\bj_t\rangle) \in L^1([0,T])$; 
    \item $(\rho,\bj)$ satisfies
    \begin{equation*}
      \int_{\R^d}\varphi(x)\dd \rho_t(x) -  \frac{1}{2}\int_{0}^t \iint_{\Rddiag}\onabla\varphi(x,y) \dd \bj_s(x,y)  \dd s = \int_{\Rd}\varphi(x)\dd \rho_0(x)  \qquad \text{for a.e.~$t \in [0,T]$};
    \end{equation*}
  \end{listi}
in this case, we write $(\rho,\bj)\in\NCET$.
\end{definition}
For the sake of completeness, we specify that absolute continuity of $\rho$ is guaranteed by the integrability of the flux divergence. The nonlocal continuity equation above is the driving dynamics we consider in this work, already considered, e.g., in~\cite{esposito2021nonlocal,esposito2022class}, though the concept of solution is slightly different, as it is given using duality between $C_0$ and signed Radon measures. The latter is more suitable when dealing with the discrete-to-continuum limit (Section~\ref{sec:discrete_to_continuum}).

As discussed in the introduction, in order to specify the evolution on graphs we need to suitably define the flux, $F^\Phi$, which is an edge-based quantity, as well as the velocities. As the mass is solely vertex-based, one has to interpolate the quantities located at the vertices to obtain an edge-based quantity. As a byproduct, the relation between the flux and the velocity can vary, depending on the application. We specify below interpolations and fluxes used in this manuscript, following~\cite{esposito2022class}, allowing for a linear dependence on the weight function, $\eta$, in the flux.
\begin{definition}[Admissible flux interpolation]\label{def:admissible_interpolation}  A measurable function $\Phi: \mathbb{R}^3 \rightarrow \mathbb{R}$ is called an admissible flux interpolation provided that the following conditions hold:
\begin{enumerate}[label=(\roman*)]
    \item\label{ass:interp_deg} $\Phi$ satisfies
$$
\Phi(0,0 ; v)=\Phi(a, b ; 0)=0, \quad \text { for all } a, b, v \in \mathbb{R} ;
$$
 \item\label{ass:interp_lip} $\Phi$ is argument-wise Lipschitz in the sense that, for some $L_{\Phi}>0$, any a, $b, c, d, v, w \in$ $\mathbb{R}$, it holds
$$
\begin{aligned}
|\Phi(a, b ; w)-\Phi(a, b ; v)| & \leq L_{\Phi}(|a|+|b|)|w-v|; 
\\
|\Phi(a, b ; v)-\Phi(c, d ; v)| & \leq L_{\Phi}(|a-c|+|b-d|)|v|;
\end{aligned}
$$
\item $\Phi$ is positively one-homogeneous in its first and second arguments, that is, for all $\alpha>0$ and $(a, b, w) \in \mathbb{R}^3$, it holds
$$
\Phi(\alpha a, \alpha b ; w)=\alpha \Phi(a, b ; w).
$$
\end{enumerate}
\end{definition}

\begin{example}\label{ex:example_interpolations}
In \cite{esposito2021nonlocal}, the authors consider the upwind interpolation given by 
\[
\Phi_{\text {upwind }}(a, b ; w)=a w_{+}-b w_{-}, \qquad \mbox{ for } (a,b,w)\in\R^3.
\]
Another example is provided by the mean multipliers
\[
\Phi_{\text {prod }}(a, b ; w)=\phi(a, b)w, \qquad \mbox{ for } (a,b,w)\in\R^3,
\]
where common choices for $\phi$ include: $\phi(a,b) = \frac{a+b}{2}$, or $\phi(a,b) = \max\{a,b\}$.
\end{example}
The fluxes considered are defined as follows.
\begin{definition}[Admissible flux]\label{def:flux} Let $\Phi$ be an admissible flux interpolation, and let $\rho \in \mathcal{M}_{\mathrm{TV}}(\mathbb{R}^d)$, $w \in \mathcal{V}^{\mathrm{as}}(\Rddiag):=\left\{v: \Rddiag \rightarrow \mathbb{R}| v(x,y)=-v(y,x)\right\}$, and $\eta:\Rddiag\to \R$ measurable. Furthermore, take a reference measure $\lambda \in \mathcal{M}^+(\mathbb{R}^{2 d})$ such that $\rho \otimes \mu, \mu \otimes \rho \ll \lambda$. Then, the admissible flux $F^{\Phi}[\mu, \eta ; \rho, w] \in \mathcal{M}(\Rddiag)$ at $(\rho, w)$ is defined by 
$$
\mathrm{d} F^{\Phi}[\mu, \eta ; \rho, w]=\Phi\left(\frac{\mathrm{d}(\rho \otimes \mu)}{\mathrm{d} \lambda}, \frac{\mathrm{d}(\mu \otimes \rho)}{\mathrm{d} \lambda} ; w\right) \eta\, \mathrm{d} \lambda .
$$
\end{definition}

\begin{remark}\label{rem:flux_eta_dependent}
Note that, in view of the one-homogeneity of $\Phi$, the definition of admissible flux is independent of the choice of $\lambda$ as long as the absolute continuity assumption is satisfied. For instance, one could choose $\lambda=|\rho| \otimes \mu+\mu \otimes|\rho|$. Differently from~\cite{esposito2021nonlocal,esposito2022class}, in this article $\eta$ evolves in time and, consequently, so does the set of edges $G$. For this reason, we include the weight function in the flux, which is now a measure on $\Rddiag$ and not on $G$. We leave to a future investigation possible nonlinear relations between the flux and the weight function $\eta$. 
\end{remark}

As mentioned in the introduction, the evolution of $\eta$ is determined by a function $H:[0,T]\times \Rddiag \times C_b(\Rddiag)\times \cMtv(\R^d) \to \R$ given by
\[
H(t,x,y,\eta_t,\rho_t):=\w_t[\rho_t](x,y)-\eta_t(x,y),
\]
where $\w:[0,T]\times \cMtv(\R^d)\times \Rddiag  \to  \R$ satisfies properties we postpone to Section~\ref{sec:well_posed}.
We focus on the system
\begin{equation}
   \begin{split} 
    \partial_t \rho_t & = - \onabla \cdot F^\Phi[\mu, \eta_t ; \rho_t,V_t[\rho_t]],  
    \\
    \partial_t \eta_t & =  \w_t[\rho_t] -\eta_t, 
    \end{split}\tag{Co-NCL}
\end{equation}
for given initial data $\rho^0 \in \cMtv^M(\R^d)$ and $\eta^0\in C_{b}(\Rddiag)$. A solution to~\eqref{eq:ivp} is intended as follows, complementing Definition~\ref{def:nce-flux-form} with the coupled equation for $\eta$.
\begin{definition}[Solution to the initial value problem \eqref{eq:ivp}]\label{def:sol_to_ivp}
   Given an admissible flux interpolation $\Phi$, a velocity field $V:[0,T]\times \cMtv(\R^d)\times \Rddiag \to \mathcal{V}^{as}(\Rddiag)$, and function $\w:[0,T]\times \cMtv(\R^d)\times \Rddiag \to \R$, a pair $(\rho, \eta): [0,T] \to \cMtv(\R^d)\times  C_{b}(\Rddiag)$ is a solution to the initial value problem \eqref{eq:ivp} if, for any $\varphi \in C_0(\Rd)$, 
   \begin{enumerate}[label=(\roman*)]
       \item $\rho \in AC([0,T], \cMtv(\R^d)),\ \eta \in AC([0,T],  C_{b}(\Rddiag))$;\label{cond:sol_1}
       \item the maps $t \mapsto \langle \varphi,\onabla \cdot F^\Phi[\mu, \eta_t ; \rho_t,V_t[\rho_t]]\rangle$ and $t \mapsto  \w_t[\rho_t] -\eta_t$ belong to $L^1([0,T])$;  \label{cond:sol_2}
       \item for a.e. $t \in [0,T]$, every $(x,y)\in\Rddiag$,  for any $\varphi\in C_0(\R^d)$, the following conditions hold \label{cond:sol_3}
       \begin{align}
           \int_{\R^d}\varphi\dd \rho_t  & = \int_{\Rd}\varphi\dd \rho_0 + \frac{1}{2} \int_{0}^t \iint_{\Rddiag}\onabla\varphi\dd F^\Phi[\mu,\eta_s,\rho_s;V_s[\rho_s]] \dd s\label{eq:rho_evolution}
           \\
           \eta_t(x,y) &= \eta_0(x,y)  + \int_0^t  \left(\w_s[\rho_s](x,y) -\eta_s(x,y)\right) \,\dd s.
           \label{eq:eta_evolution}
       \end{align}
   \end{enumerate}
\end{definition}

In the following, for any curve $\gamma \in C([0,T],S)$, for some normed space $(S, \norm{\cdot}_S)$, we will write 
\[
\norm{\gamma}_{\infty,S}:= \sup_{t \in [0,T]}\norm{\gamma_t}_{S} \ . 
\]
We denote by $\|\cdot\|_\infty$ the usual sup-norm when this does not create confusion.

\subsection{A priori properties}

Our strategy to prove well-posedness for~\eqref{eq:ivp} relies on the application of the Banach fixed-point theorem. In the next proposition we collect properties of solutions important for the definition of the solution map, under assumptions milder than those needed later on the velocity fields.  

\begin{proposition}\label{prop:int_supp_mass_preservation} Let $\Phi$ be an admissible flux interpolation, $\rho_0 \in \cMtv^M(\mathbb{R}^d)$, $\eta_0\in C_b(\Rddiag)$. Assume $\w:[0,T]\times \cMtv(\R^d)\times \Rddiag  \to  \R$ and $V:[0, T]\times \cMtv(\R^d) \times \Rddiag \rightarrow \mathcal{V}^{\mathrm{as}}(\Rddiag)$ satisfy, for some $C_V>0$ and $C_\w>0$,
\begin{subequations}
\begin{equation}\label{eq:assumption_weak_compressibility_V}
\int_0^T \sup_{\rho \in \cMtv(\R^d)} \sup _{x \in \mathbb{R}^d} \int_{\mathbb{R}^d \backslash\{x\}}\left|V_s[\rho](x, y)\right| \mathrm{d} \mu(y)\dd s \leq C_V, 
\end{equation}
\begin{equation}\label{eq:assumption_weak_compressibility_w}
    \int_0^T \sup_{\rho \in \cMtv(\R^d)}\sup_{(x,y) \in \Rddiag}|\w_s[\rho](x,y)| \dd s\le C_\w.
\end{equation}
\end{subequations}
Suppose the map $(x,y)\in\Rddiag  \mapsto \w_t[\sigma](x,y)$ is continuous, for any $t\in[0,T]$ and $\sigma \in \cMtv(\Rd)$. For a pair $(\rho, \eta):[0,T] \to  \cMtv(\R^d) \times C_{b}(\Rddiag)$ satisfying~\ref{cond:sol_3} in Definition~\ref{def:sol_to_ivp}, the following properties hold:
\begin{enumerate}
    \item  for any $\varphi \in C_0(\R^d)$ the maps 
    \begin{align*}
       & t \mapsto \iint_{\Rddiag}\onabla\varphi\dd F^\Phi[\mu,\eta_t,\rho_t;V_t[\rho_t]] \in L^1([0,T]),
       \\
       & t \mapsto  \w_t[\rho_t](x,y)- \eta_t(x,y) \in L^1([0, T]), 
    \end{align*} and $\rho \in L^\infty([0, T] ; \cMtv(\R^d))$, $\eta\in L^\infty([0,T],C_{b}(\Rddiag))$ (flux integrability and time boundedness);
    \item $\rho_t[\mathbb{R}^d]=\rho_0[\mathbb{R}^d]$ for all $t \in[0, T]$ (mass preservation);
    \item $\rho \in AC([0,T], \M^{M}_{TV}(\R^d))$ and $\eta \in AC([0,T], C_b(\Rddiag))$ (absolute continuity);
    \item if $\operatorname{supp} \rho_0 \subseteq \operatorname{supp} \mu$, then $\operatorname{supp} \rho_t \subseteq \operatorname{supp} \mu$ for a.e. $t \in[0, T]$ (support inclusion).
\end{enumerate}
\end{proposition}
\begin{proof}
\textit{Time boundedness and flux integrability} -  We start proving these properties hold for $\eta$. From~\eqref{eq:eta_evolution} we have
\begin{align*}
     |\eta_t(x,y)| & \leq |\eta_0(x,y)| + \int_0^t  (|\w_s[\rho_s](x,y)| + |\eta_s(x,y)|)\dd s,
\end{align*}
whence 
\[
 \|\eta_t\|_{\infty} \leq \|\eta_0\|_{\infty} + \int_0^t\sup_{\rho \in \cMtv(\R^d)}\sup_{(x,y)\in\Rddiag}|\w_s[\rho](x,y)|\dd s + \int_0^t \|\eta_s\|_{\infty}\dd s.
\]

Then, an application of Gr\"onwall's inequality yields 
\begin{align}\label{eq:eta_time_bound}
   \|\eta_t\|_{\infty} & \leq \left(\|\eta_0\|_{\infty} + \int_0^T \sup_{\rho \in \cMtv(\R^d)}\sup_{(x,y) \in \Rddiag}|\w_s[\rho](x,y)| \dd s\right) e^{T}.
\end{align}
By taking the supremum for $t\in[0,T]$, using assumption~\eqref{eq:assumption_weak_compressibility_w} on $\w$, yields the result for $\eta$, meaning $\eta \in L^ \infty([0,T],C_{b}(\Rddiag))$. Next, we prove an analogous result for $\rho$.

\noindent
By definition of admissible flux and nonlocal divergence, for any $\varphi \in C_0(\Rd) $, we have
\begin{align*}
   \langle\varphi,\onabla \cdot F^{\Phi}\left[\mu,\eta_t ; \rho_t, V_t[\rho_t]\right]\rangle&=-\frac{1}{2} \iint_{\Rddiag} \onabla \varphi(x,y) \mathrm{d} F^{\Phi}\left[\mu, \eta_t ; \rho_t, V_t[\rho_t]\right](x, y) 
   \\
   &= - \frac{1}{2} \iint_{\Rddiag} \onabla \varphi(x,y) \Phi\left(\frac{\mathrm{d}\left(\rho_t \otimes \mu\right)}{\mathrm{d} \lambda}, \frac{\mathrm{d}\left(\mu \otimes \rho_t\right)}{\mathrm{d} \lambda} ; V_t[\rho_t]\right) \eta_t(x,y) \mathrm{d} \lambda(x,y) .
\end{align*}
Using~\ref{ass:interp_deg},~\ref{ass:interp_lip} in Definition~\ref{def:admissible_interpolation},~\cite[Theorem 6.13]{rudin1987complex}, the boundedness of $\eta$ and the antisymmetry of the velocity field, we have, for any $\varphi \in C_0(\R^d)$ and $t\in[0,T]$:
\begin{align}\label{eq:flux_bound}
\int_0^t\left|\langle\varphi,\onabla \cdot F^{\Phi}\left[\mu,\eta_s ; \rho_s, V_s[\rho_s]\right]\rangle\right| \mathrm{d} s & \leq \frac{L_{\Phi}\|\eta\|_{\infty}}{2}\int_0^t \!\!\iint_\Rddiag |\onabla\varphi|\left|V_s[\rho_s]\right| \left(\mathrm{d}\left|\rho_s\right| \otimes \mu+\mathrm{d} \mu \otimes\left|\rho_s\right|\right) \mathrm{d} s \nonumber 
\\
& \leq L_{\Phi}\|\eta\|_{\infty}\norm{\varphi}_\infty \int_0^t \iint_\Rddiag\left|V_s[\rho_s](x,y)\right| \mathrm{d} \mu(y) \mathrm{d}\left|\rho_s\right|(x) \mathrm{d} s \nonumber 
\\
& \leq L_{\Phi}\|\eta\|_{\infty}\norm{\varphi}_\infty \int_0^t \overline{v}_s\left|\rho_s\right|[\mathbb{R}^d] \mathrm{d} s,
\end{align}
where we set $\overline{v}_s:=\sup_{\rho \in \cMtv(\R^d)}\sup_{x\in\R^d}\int_{\R^d\setminus\{x\}}\left|V_s\left[\rho_s\right](x,y)\right|\dd\mu(y)$. From \eqref{eq:rho_evolution} in Definition \ref{def:sol_to_ivp}, \eqref{eq:flux_bound}, and the definition of the total variation norm, we infer 
\begin{align*}
    |\rho_t|[\R^d]&\leq  |\rho_0|[\R^d] + L_\Phi \|\eta\|_{\infty}\int_0^T \overline{v}_s|\rho_s|[\R^d] \dd s.
\end{align*}
An application of Gr\"onwall's inequality and assumption \eqref{eq:assumption_weak_compressibility_V} provide
\begin{equation}\label{eq:rho_time_bound}
    |\rho_t |[\R^d] \leq  |\rho_0 |[\R^d] e ^{L_\Phi\|\eta\|_{\infty} C_V} < \infty.
\end{equation}
Hence, we can conclude the time integrability of the flux and that $\rho \in L^\infty([0,T],\M_{TV}(\R^d))$.

\textit{Mass preservation} - Let $\varphi^{R}(x):= e^{-|x|^2/R^2}$. Note that $\varphi^{R} \in C^\infty_0(\Rd)$ and $\varphi^R(x) \to 1$ pointwise as $R\to \infty$. Using Definition~\ref{def:sol_to_ivp}~\ref{cond:sol_3}, we infer
\begin{align}\label{eq:last_change}
    \bigg|\int_{\Rd}\varphi^R(x) \dd \rho_t(x) & - \int_{\Rd}\varphi^R(x) \dd \rho_0(x) \bigg|\nonumber
    \\
    &\leq \frac{1}{2}\int_0^t\bigg|\iint_{ \Rddiag}(\varphi^R(y)-\varphi^R(x))\eta_s\Phi\left(\frac{\mathrm{d}\left(\rho_t \otimes \mu\right)}{\mathrm{d} \lambda}, \frac{\mathrm{d}\left(\mu \otimes \rho_t\right)}{\mathrm{d} \lambda} ; V_t[\rho_t]\right)\mathrm{d} \lambda(x,y)\bigg|\dd s\nonumber
    \\
    & \leq L_\Phi \norm{\eta}_\infty \int_0^t \iint_{ \Rddiag}|\varphi^R(y)-\varphi^R(x)|\,|V_s[\rho_s]|\dd|\rho_s|(x)\dd\mu(y)\dd s.
\end{align}
The previous integrals can be bounded as follows:
\begin{align*}
  \int_0^t \iint_{ \Rddiag}|\varphi^R(y)-\varphi^R(x)|\,|V_s[\rho_s]|\dd|\rho_s|(x)\dd\mu(y)\dd s &\leq 2\int_0^t \iint_{ \Rddiag}|V_s[\rho_s]|\dd|\rho_s|(x)\dd\mu(y)\dd s
  \\
  & \leq 2 \int_0^t \overline{v}_s |\rho_s|[\R^d] \dd s
  \\
  & \leq 2C_V |\rho_0 |[\R^d] e ^{L_\Phi\|\eta\|_{\infty} C_V}T
\end{align*}
where in the last inequality we used the bound \eqref{eq:rho_time_bound}. Hence, by means of the Dominated Convergence Theorem, as we let $R\to \infty$ in~\eqref{eq:last_change}, we obtain 
\[
\bigg|\int_{\Rd}\varphi^R(x) \dd \rho_t(x) - \int_{\Rd}\varphi^R(x) \dd \rho_0(x) \bigg| \longrightarrow 0 \text{ as } R\to \infty \ .
\]
Thus, noting that $\rho_t[\Rd]=\lim_{R \to \infty}\int_{\Rd}\varphi^R(x)\dd\rho_t(x)$, which follows again from the Dominated Convergence theorem, as well as $\rho_0[\Rd]=\lim_{R \to \infty}\int_{\Rd}\varphi^R(x)\dd\rho_0(x)$, we conclude the proof.

\textit{Absolute continuity} - This is a direct consequence of the flux integrability proven by exploiting~\eqref{eq:rho_evolution}~and~\eqref{eq:eta_evolution}.

\textit{Support inclusion for $\rho$} -  Let $A=\Rd \setminus \supp \mu$ and  $(\rho_t)_{t\in[0,T]}$ be a solution to \eqref{eq:rho_evolution}. As we want to evaluate $|\rho_t|(A)$, we consider test functions $\varphi\in C_c(A)$, due to~\cite[Proposition 1.47]{ambrosio2000functions}. Using~\eqref{eq:rho_evolution}, we estimate:

\begin{align*}
    \bigg|\int_{A}\varphi(x) \dd \rho_t(x) \bigg| &\leq  \bigg|\int_{A}\varphi(x) \dd \rho_0(x) \bigg|
    \\
    &\quad +\frac{L_\Phi}{2}\int_0^t \iint_{ \Rddiag \cap\operatorname{supp} \mu \times  A}|\varphi (y)|\, |\eta_s|\,|V_s[\rho_s](x,y)|\dd\mu(x)\dd|\rho_s|(y) 
    \\
    & \quad +\frac{L_\Phi}{2}\int_0^t \iint_{ \Rddiag \cap A \times \operatorname{supp} \mu }|\varphi (x)|\, |\eta_s|\, |V_s[\rho_s](x,y)|\dd|\rho_s|(x) \dd\mu(y)
    \\
    &\leq \bigg|\int_{A}\varphi(x) \dd \rho_0(x) \bigg|+L_\Phi\norm{\eta}_\infty\int_0^t \iint_{\Rddiag \cap A \times \operatorname{supp} \mu }|\varphi (x)|\,  |V_s[\rho_s](x,y)|\dd|\rho_s|(x) \dd\mu(y)
    \\
    & \leq \bigg|\int_{A}\varphi(x) \dd \rho_0(x) \bigg|+L_\Phi\norm{\eta}_\infty\int_0^t \int_{A}|\varphi(x)|\dd |\rho_s|(x) \overline{v}_s\dd s\ .
\end{align*}
Taking the supremum over all $\varphi \in C_c(A)$ such that $\norm{\varphi}_\infty\leq1$ we infer, by~\cite[Proposition 1.47]{ambrosio2000functions} and Gr\"{o}nwall's inequality, that 
\begin{align*}
    |\rho_t|[A] \leq e^{C_V L_\Phi\|\eta\|_{\infty}}|\rho_0|[A]=0
\end{align*}
by definition of $A$. 
\end{proof}

\section{Well-posedness of the co-evolving non-local continuity equation}\label{sec:well_posed}

In this section we show that the system \eqref{eq:ivp} is well-posed by means of a Banach  fixed-point argument. From now on we fix $\rho_0\in\cMtv^M(\Rd)$, $\eta_0\in C_b(\Rddiag)$, and assume $\Phi$ is an admissible flux interpolation according to Definition~\ref{def:admissible_interpolation}. Let us denote
\[
\ACt:=AC([0,T]; \M_{TV}^M(\R^d)) \times AC([0,T]; C_b(\Rddiag)),
\]
and equip this space with the metric defined by 
\[
d_\infty((\rho^1,\eta^1),(\rho^2,\eta^2)) := \| \rho^1-\rho^2\|_{\infty,\M_{TV}(\R^d)} + \|\eta^1-\eta^2 \|_{\infty, C_b(\Rddiag)},
\]
where
\begin{align*}
    &\| \rho^1-\rho^2\|_{\infty,\M_{TV}(\R^d)}=\sup_{t\in[0,T]}\|\rho_t^1-\rho^2_t\|_{\TV},\\
    &\|\eta^1-\eta^2 \|_{\infty,C_b(\Rddiag)}=\sup_{t\in[0,T]}\|\eta^1_t-\eta^2_t\|_{\infty}.
\end{align*}

 We assume the velocity field $V:[0,T]\times \M_{TV}^M(\R^d) \times \Rddiag \to \V^{\mathrm{as}}(\Rddiag)$ to fulfil, for a constant $C_V>0$, the uniform compressibility assumption
\begin{equation}\label{eq:velocity_bound}
  \sup _{t \in[0, T]} \sup _{\rho \in \M_{TV}^M(\mathbb{R}^d)} \sup _{x \in \mathbb{R}^d} \int_{\mathbb{R}^d \backslash\{x\}}\left|V_t[\rho](x, y)\right| \mathrm{d} \mu(y)  \leq C_V.
\end{equation}
Assumption~\eqref{eq:velocity_bound} is an $L^\infty$ bound for the nonlocal divergence needed later in our argument to obtain a contraction. We note this is a stronger requirement with respect to time, than~\eqref{eq:assumption_weak_compressibility_V}, which was used to obtain the time-continuity of solutions in Proposition~\ref{prop:int_supp_mass_preservation}.

\begin{remark}
    We refer to~\eqref{eq:velocity_bound} as \textit{uniform compressibility} assumption in relation to the terminology used for the continuity equation $\partial_t\rho_t+\nabla\cdot(b_t\rho_t)=0$, for a vector field $b:[0,T]\times\Rd\to\Rd$. In this context an $L^\infty$-bound is required on $\nabla\cdot b$, which is replaced in our framework by~\eqref{eq:velocity_bound}. We refer the reader to~\cite[Remark~3.6 and Section 4]{esposito2022class} for further details. 
\end{remark}
We fix $\w:[0,T]\times \M_{TV}^M(\R^d) \times \Rddiag \to \R$ satisfying the following conditions: 
\begin{enumerate}[label=$(\bm{\w}\arabic*)$]
    \item the map $(x,y)\in\Rddiag  \mapsto \w_t[\cdot](\cdot,x,y)$ is continuous and $(t\mapsto\w_t[\cdot](\cdot,\cdot))\in L^1([0,T])$;\label{ass:w_continuous}
    \item for any $\rho, \sigma \in \cMtv(\R^d)$ there exists a constant $L_\w\geq 0$ such that\label{ass:omega_lip}
     \begin{align*}
         \sup_{t \in [0,T]}\sup_{x,y \in \Rddiag} | \w_t[\sigma](x,y) - \w_t[\rho](x,y)| & \leq L_\w\norm{\sigma - \rho}_{TV}; 
     \end{align*}
     \item $\w$ is bounded, that is, there exists a constant $C_\w>0$ such that \label{ass:omega_bounded}
     \begin{equation*}
    \sup_{t \in [0,T]}\sup_{\rho \in \cMtv(\R^d)}\sup_{x,y \in \Rddiag}\big|\w_t[\rho](x,y)\big| \le C_\w. 
\end{equation*}
\end{enumerate}
We define the solution maps $\S:=(\S_V,\S_\w):\ACt\to \ACt$ by
\begin{subequations}\label{eq:sol_maps}
\begin{align}
    & \S_V(\rho,\eta)(t) := \int_{\Rd}\varphi\dd \rho_0 +  \frac{1}{2}\int_{0}^t \iint_{\Rddiag}\onabla\varphi\dd F^\Phi[\mu,\eta_s,\rho_s;V_s[\rho_s]] \dd s ,
    \\
    &
    \S_\w(\rho,\eta)(t)(x,y):=\eta^0(x,y) + \int_0^t   \w_s[\rho_s](x,y) - \eta_s(x,y) \dd s  .
\end{align}
\end{subequations}
for any  $\varphi \in C_0(\R^d)$, $t \in [0,T]$ and $(x,y) \in \Rddiag$. We observe that the map $\S: \ACt \to \ACt$ is well-defined due to our assumptions on the functions $V$, \eqref{eq:velocity_bound}, and $\w$, \ref{ass:w_continuous} and \ref{ass:omega_bounded}, taking into account Proposition~\ref{prop:int_supp_mass_preservation}. 

\begin{remark}
     In Section 5, we shall compare different solutions $(\rho,\eta),(\widetilde{\rho},\widetilde{\eta})$ of \eqref{eq:ivp}, with potentially different base measures $\mu, \widetilde{\mu}$. Note that we can take $\lambda:= |\rho| \otimes \mu + \mu \otimes |\rho| +|\widetilde{\rho}| \otimes \widetilde{\mu} + \widetilde{\mu} \otimes |\widetilde{\rho}| $ in Definition \ref{def:flux} in order to be able to compare fluxes associated with $(\rho,\eta)$ and $(\widetilde{\rho},\widetilde{\eta})$.
\end{remark}
First, we show the following contraction with respect to $d_\infty$, crucial for the application of the Banach fixed-point theorem. 
\begin{lemma}\label{lemma:contraction}
Let $V:[0,T] \times \M_{TV}^M(\R^d) \times \Rddiag \to \V^{\mathrm{as}}(\Rddiag)$ satisfy assumption \eqref{eq:velocity_bound} and assume there is a constant $L_V \geq 0$ such that, for all $t \in [0,T]$ and all $\rho, \sigma \in \M_{TV}^M(\R^d)$,
\begin{equation}\label{eq:ass_Lip_rho}
         \sup _{x \in \mathbb{R}^d} \int_{\mathbb{R}^d \backslash\{x\}}\left|V_t[\rho](x, y)-V_t[\sigma](x, y)\right| \mathrm{d} \mu(y)  \leq L_V\|\rho-\sigma\|_{\mathrm{TV}} \ .
\end{equation}
Assume $\w$ satisfies assumptions \ref{ass:w_continuous}---\ref{ass:omega_bounded} and that the solution map, $\S$, is defined as in \eqref{eq:sol_maps}. Then, for any $(\rho,\eta), (\widetilde{\rho}, \widetilde{\eta}) \in \ACt$, the following contraction estimate holds
\[
d_\infty\big(\S(\rho,\eta),\S(\widetilde{\rho},\widetilde{\eta})\big) \leq \kappa(T) T d_\infty((\rho,\eta),(\widetilde{\rho}, \widetilde{\eta})),
\]
where $\kappa(T):=\kappa(M,L_\Phi,L_V,C_\w,\|\eta_0\|_{\infty},T)\ge0$. In particular, there exists a $T^*>0$ solving $\kappa(T^*)T^*=1$ such that we have a unique solution $(\rho,\eta)$ to \eqref{eq:ivp} on $[0,T]$, for $T<T^*$, and $(\rho(0),\eta(0)) = (\rho_0,\eta_0)\in \M^M_{TV}(\R^d) \times C_{b}(\Rddiag)$.
\end{lemma}
\begin{proof}
Let $(\rho,\eta)$ and $(\widetilde{\rho},\widetilde{\eta})$ belong to $\ACt$. Upon using our assumptions on the flux $\Phi$ in Definition~\ref{def:flux} we obtain 
\begin{align*}
|S_V(\rho, \eta)(t)-S_V(\widetilde{\rho}, \widetilde{\eta})(t)|
&\leq \frac{1}{2}\int_0^t  \int_{\Rddiag}\bigg|(\varphi(y) - \varphi(x)) \bigg( \eta_s \Phi\left(\frac{\dd \rho_s\otimes \dd\mu}{\dd \lambda }, \frac{\dd\mu \otimes \dd\rho_s}{\dd \lambda } ; V_s\left[\rho_s\right]\right) 
\\
& \qquad \qquad -\widetilde{\eta}_s \Phi\left(\frac{\dd \widetilde{\rho}_s\otimes \dd\mu}{\dd \lambda }, \frac{\dd\mu \otimes \dd\widetilde{\rho}_s}{\dd \lambda };V_s\left[\widetilde{\rho}_s\right]\right)\bigg) \bigg|\dd \lambda   \dd s\\
&\leq  \frac{1}{2}\int_0^t\int_{\Rddiag}\bigg|(\varphi(y) - \varphi(x))\eta_s\bigg(\Phi\left(\frac{\dd \rho_s\otimes \dd\mu}{\dd \lambda}, \frac{\dd\mu \otimes \dd\rho_s}{\dd \lambda } ; V_s\left[\rho_s\right]\right) 
\\
&\qquad \qquad -\Phi\left(\frac{\dd \widetilde{\rho}_s\otimes \dd\mu}{\dd \lambda }, \frac{\dd\mu \otimes \dd\widetilde{\rho}_s}{\dd \lambda};V_s\left[{\rho}_s\right]\right)\bigg)\bigg| \dd\lambda\dd s
\\
&\quad + \frac{1}{2} \int_0^t\int_{\Rddiag}\bigg|(\varphi(y) - \varphi(x))\left(\eta_s-\widetilde{\eta}_s\right) 
\\
&  \qquad \qquad  \times \Phi\left(\frac{\dd \widetilde{\rho}_s\otimes \dd\mu}{\dd \lambda}, \frac{\dd\mu \otimes \dd\widetilde{\rho}_s}{\dd \lambda};V_s\left[{\rho}_s\right]\right)\bigg| \dd\lambda\dd s
\\
&\quad + \frac{1}{2} \int_0^t\int_{\Rddiag} \bigg| (\varphi(y) - \varphi(x))\widetilde{\eta}_s\bigg(\Phi\left(\frac{\dd \widetilde{\rho}_s\otimes \dd\mu}{\dd \lambda}, \frac{\dd\mu \otimes \dd\widetilde{\rho}_s}{\dd \lambda};V_s\left[{\rho}_s\right]\right)
\\
& \qquad \qquad -\Phi\left(\frac{\dd \widetilde{\rho}_s\otimes \dd\mu}{\dd \lambda}, \frac{\dd\mu \otimes \dd\widetilde{\rho}_s}{\dd \lambda};V_s\left[{\widetilde{\rho}}_s\right]\right)\bigg) \bigg| \dd\lambda d s\\
&\leq \frac{L_{\Phi}}{2}  \int_0^t \int_{\Rddiag}|\varphi(y) - \varphi(x)|\left|V_s\left[\rho_s\right]\right| |\eta_s(x,y)|(\mathrm{d}\left(\left|\rho_s - \widetilde{\rho}_s\right| \otimes \mu\right)
\\
&\qquad\qquad \qquad +\mathrm{d}\left(\mu \otimes\left|\rho_s-\widetilde{\rho}_s\right|\right))\mathrm{d} s
\\
&\qquad + \frac{L_\Phi}{2}  \int_0^t \int_{\Rddiag}|\varphi(y) - \varphi(x)||V_s[{\rho}_s]| 
\\
& \qquad \qquad \qquad \qquad \times|\eta_s-\widetilde{\eta}_s|(\mathrm{d}\left(\mu \otimes\left|\widetilde{\rho_s}\right|\right)+\mathrm{d}\left(|\widetilde{\rho_s}|\otimes \mu \right))\dd s
\\
&
\qquad +\frac{L_\Phi}{2}  \int_0^t\int_{\Rddiag}|\varphi(y) - \varphi(x)| | \widetilde{\eta}_s|\left|V_s\left[\rho_s\right]-V_s\left[\widetilde{\rho_s}\right]\right| 
\\
&\qquad \qquad \qquad \qquad\times\left(\mathrm{d}\left(\left|\widetilde{\rho_s}\right| \otimes \mu\right)+\mathrm{d}\left(\mu \otimes\left|\widetilde{\rho_s}\right|\right)\right)   \dd s
\\
& 
:= I + II + III 
\end{align*}
We then have the following estimates for the terms $I,II,III$. Starting with $I$, using antisymmetry of $V$ and our assumption \eqref{eq:velocity_bound} we have 
\begin{align*}
 I& \le {L_{\Phi}} \norm{\varphi}_\infty \int_0^t \int_{\Rddiag}\left|V_s\left[\rho_s\right]\right||\eta_s |\left(\mathrm{d}\left(\left|\rho_s - \widetilde{\rho}_s\right| \otimes \mu\right)+\mathrm{d}\left(\mu \otimes\left|\rho_s-\widetilde{\rho}_s\right|\right)\right) \mathrm{d} s
 \\
 &\leq 2{L_{\Phi}}\norm{\varphi}_\infty \norm{\eta}_{\infty,C_{b}(\Rddiag)}  \int_0^t \int_{\Rddiag}\left|V_s\left[\rho_s\right]\right| \left(\mathrm{d}\left(  \mu\otimes\left|\rho_s - \widetilde{\rho}_s\right|\right)(y,x)\right) \mathrm{d} s 
 \\
 & \leq 2{L_{\Phi}} \norm{\varphi}_\infty\norm{\eta}_{\infty,C_{b}(\Rddiag)}  \int_0^t \int_{\R^d} \dd|\rho_s -\widetilde{\rho}_s|(x)\left(\sup _{\rho \in \M_{TV}^M\left(\mathbb{R}^d\right)}  \sup _{x \in \mathbb{R}^d}\int_{\R^d\setminus\{x\}}|V_s[\rho_s](x, y)| \dd\mu(y)\right)\dd s
 \\
 & \leq 2L_\Phi \norm{\varphi}_\infty C_V \norm{\eta}_{\infty,C_{b}(\Rddiag)}\|\rho-\widetilde{\rho}\|_{\infty, \cMtv(\R^d)}T
\end{align*}
For the term $II$, using antisymmetry of $V$ and \eqref{eq:velocity_bound} we get      
\begin{align*}
II & \le  L_\Phi\norm{\varphi}_\infty  \int_0^t \int_{\Rddiag}|V_s[{\rho}_s]| 
|\eta_s-\widetilde{\eta}_s|(\mathrm{d}\left(\mu \otimes\left|\widetilde{\rho}_s\right|\right)+\mathrm{d}\left(|\widetilde{\rho}_s|\otimes \mu \right))\dd s
\\
& \leq  2  L_\Phi\norm{\varphi}_\infty\|\eta-\widetilde{\eta} \|_{\infty, C_{b}(\Rddiag)} \int_0^t \int_{\Rddiag} |V_s[{\rho}_s](x,y)|| \dd \mu(y) \dd |\widetilde{\rho}_s|(x) \dd s
\\
&\leq 2 L_\Phi\norm{\varphi}_\infty C_V M \|\eta-\widetilde{\eta} \|_{\infty, C_{b}(\Rddiag)}T
\end{align*}
Finally, we have the following bound for term $III$ using our assumption \eqref{eq:ass_Lip_rho} on $V$:
\begin{align*}
    III& \le L_\Phi \norm{\varphi}_\infty \int_0^t\int_{\Rddiag}|\widetilde{\eta}_s|\left|V_s\left[\rho_s\right]-V_s\left[\widetilde{\rho}_s\right]\right| 
\left(\mathrm{d}\left(\left|\widetilde{\rho}_s\right| \otimes \mu\right)+\mathrm{d}\left(\mu \otimes\left|\widetilde{\rho}_s\right|\right)\right) \mathrm{d} s
    \\
    &\leq 2  L_\Phi\norm{\varphi}_\infty \norm{\widetilde{\eta}}_{\infty,C_{b}(\Rddiag)} \int_0^t\int_{\Rddiag} \left|V_s\left[\rho_s\right](x,y)-V_s\left[\widetilde{\rho}_s\right](x,y)\right|\dd \mu(y) \dd |\widetilde{\rho}_s|(x)  \dd s
    \\
    & \leq 2L_\Phi \norm{\varphi}_\infty  L_VM \norm{\widetilde{\eta}}_{\infty,C_{b}(\Rddiag)}\|\rho-\widetilde{\rho}\|_{{\infty,\cMtv}(\R^{d})}T.
\end{align*}
Thus, taking the supremum over all $\varphi \in C_0(\R^d)$ with $\norm{\varphi}_\infty\leq1$ and recalling the definition of the $\TV$-norm we have the following estimate
\begin{align}\label{eq:stability_rho}
\|\S_V(\eta,\rho) - \S_V(\widetilde{\eta},\widetilde{\rho})\|_{\infty,\M_{TV}^M(\R^d)} &\leq   2L_\Phi  C_V \norm{\eta}_{\infty,C_{b}(\Rddiag)}\|\rho-\widetilde{\rho}\|_{\infty, \M^M_{TV}(\R^d)}T \nonumber
 \\
 & \qquad +  2 L_\Phi C_V M \|\eta-\widetilde{\eta} \|_{\infty, C_{b}(\Rddiag)}T
 \\
 &\qquad +2L_\Phi   L_VM \norm{\widetilde{\eta}}_{\infty,C_{b}(\Rddiag)}\|\rho-\widetilde{\rho}\|_{{\infty,\M_{TV}^M}(\R^{d})}T.\nonumber
\end{align}
On the other hand for $(x,y) \in \Rddiag$ and $t \in [0,T]$ we have
\begin{equation}\label{eq:eta_stability}
\begin{split}
|\S_\w(\rho,\eta)(x,y)(t)  - \S_\w(\widetilde{\rho},\widetilde{\eta})(x,y)(t)  |& \leq \int_0^t|\eta_s(x, y)-\widetilde{\eta}_s(x, y)|\dd s
\\
& \qquad + \int_0^t  |\w_s[\rho_s](x,y)-\w_s[\widetilde{\rho}_s](x,y)|\dd s  \ .
\end{split}
\end{equation}
Thus, by assumption~\ref{ass:omega_lip},
\begin{equation*}
\begin{split}
 \norm{\S_\w(\rho, \eta)-\S_\w(\widetilde{\rho}, \widetilde{\eta})}_{\infty,C_{b}(\Rddiag)}
& \leq \norm{\eta-\widetilde{\eta}}_{\infty,C_{b}(\Rddiag)}T
\\
& \qquad \quad +\int_0^T\sup_{(x,y) \in \Rddiag}\left|\w_t[\rho_t](x,y)-\w_t[\widetilde{\rho}_t](x,y)\right|\dd t
\\
& \leq \bigg(\norm{\eta-\widetilde{\eta}}_{\infty,C_b(\Rddiag)}
\\
& \qquad + L_\w\norm{\rho - \widetilde{\rho}}_{\infty,\M_{TV}^M(\R^d)}\bigg) T.
\end{split}
\end{equation*}
According to Proposition~\ref{prop:int_supp_mass_preservation}, Eq.~\eqref{eq:eta_time_bound}, having assumption~\ref{ass:omega_bounded} we know
\begin{equation}\label{eq:unif_bound_eta_T}
\|\eta\|_{\infty,C_b(\Rddiag)}\le\left(\|\eta_0\|_{\infty}+C_\w T\right)e^T.
\end{equation}

Combining the previous estimates we infer 
\begin{align*}
    d_\infty((\S(\eta,\rho),\S(\widetilde{\eta},\widetilde{\rho}))&\leq  (2  L_\Phi(\norm{\eta}_{\infty,C_{b}(\Rddiag)} C_V 
    \\
    & \qquad + \norm{\widetilde{\eta}}_{\infty,C_{b}(\Rddiag)}L_VM) + L_\w ) \|\rho-\widetilde{\rho}\|_{\infty, \M_{TV}^M(\R^d)}T 
 \\
 & \qquad + (2 L_\Phi    C_V M +1 )\|\eta-\widetilde{\eta} \|_{\infty, C_b(\Rddiag)}T
\\
&
\leq \underbrace{(2  L_\Phi( C_V + L_VM)(\|{\eta}_0\|_\infty+C_\w T)e^T + L_\w )}_{:=\alpha} \|\rho-\widetilde{\rho}\|_{ \infty,\M_{TV}^M(\R^d)}T 
 \\
 &\qquad + \underbrace{(2 L_\Phi    C_V M +1 )}_{:=\beta}\|\eta-\widetilde{\eta} \|_{\infty, C_b(\Rddiag)}T
 \\
 & \leq  \max\{\alpha(T),\beta\} Td_\infty((\rho,\eta),(\widetilde{\rho},\widetilde{\eta})).
\end{align*}
When $T< \frac{1}{\max\{\alpha(T),\beta\}}$ we have that the solution map is a contraction. We check this is not an issue in case the maximum is $\alpha(T)=L_\w+2L_\Phi(C_V+L_VM)\|{\eta}_0\|_\infty e^T+2L_\Phi(C_V+L_VM)C_\w T e^T)$, which we rewrite as $\alpha(T):=\sigma+\gamma e^T+\chi T e^T$, for some positive constants $\sigma, \gamma, \chi$. The inequality $T<\frac{1}{\alpha(T)}$ is true for $T<T^*$, where $T^*=T^*(\sigma,\gamma,\chi)$ is fixed and solves $\alpha(T^*)T^*=1$. Denoting by $\kappa:= \max\{\alpha,\beta\}$ we obtain existence and uniqueness when $T< \frac{1}{\kappa}$ as a direct consequence of the Banach fixed-point theorem. 
\end{proof}
\begin{remark}
    For the sake of completeness we observe that Banach fixed-point theorem is applied to $C_T=C([0,T]; \M_{TV}^M(\R^d)) \times C([0,T]; C_b(\Rddiag))$. Absolute continuity in time follows from the fact that the flux and the map $(t\mapsto \w_t[\rho_t]-\eta_t)$ belong to $L^1([0,T])$, as proven in~Proposition~\ref{prop:int_supp_mass_preservation}.
\end{remark}

\begin{remark}\label{rem:eta_positive} Although this is not required for our result to hold, we note that, if in addition to the current assumptions on $\w$ we impose 
\[
\inf_{(x,y) \in \Rddiag}\eta_0(x,y) \geq\norm{\w_-}_\infty (e^T-1)\ , 
\]
where $\norm{\w_-}_\infty= \sup_{t \in [0,T]}\sup_{\rho \in \cMtv(\R^d)}\sup_{x,y \in \Rddiag}\big|(\w_t)_{-}[\rho](x,y)\big|$, we obtain non-negativity preservation of $\eta$ from the explicit solution
\[
   \eta_t(x,y) = e^{-t}\left(\eta_0(x,y) + \int_0^t e^{s}\w_s[\rho_s](x,y) \dd s\right).
\]
We note that in \cite{esposito2022class} only non-negative weight functions were considered, while here we allow for weights to become negative which is used in many applications, see for example~\cite{shi2019dynamics}.
\end{remark}
Now we are ready to prove existence and uniqueness of the initial value problem \eqref{eq:ivp}, which is the content of the following theorem. 
\begin{theorem}[Existence and uniqueness for~\eqref{eq:ivp}]\label{thm:well-posedness} Let $V:[0,T]\times \cMtv^M(\R^d)\times \Rddiag \to \V^{as}(\Rddiag)$ and $\w:[0,T]\times \M_{TV}^M(\R^d)\times \Rddiag \to \R$ satisfy~\eqref{eq:velocity_bound}, ~\eqref{eq:ass_Lip_rho} and~\ref{ass:w_continuous}---\ref{ass:omega_bounded}, respectively. Furthermore, let the assumptions of Lemma \ref{lemma:contraction} hold. Then there exists a unique solution $(\rho,\eta)$ to \eqref{eq:ivp} such that $(\rho_0,\eta_0)=(\rho^0,\eta^0)$.
\end{theorem}
\begin{proof}
  
     Let $\kappa$ be as in Lemma \ref{lemma:contraction}. The condition $T<\frac{1}{\kappa}$ can be also interpreted as $T<a:=\min\left\{\frac{1}{\beta}, T^*(\sigma,\gamma,\chi)\right\}$. Assume $T\ge a$, let $\tau:= a/2$, and denote $N:=\left[\frac{T}{\tau}\right]$. Then by Lemma \ref{lemma:contraction}, we know there exists a unique solution in $\ACt$ to \eqref{eq:ivp} on $[0,\tau]$. Denote this solution by $(\rho^1,\eta^1)$ and note that $(\rho^1,\eta^1) \in \mathcal{AC}_{0,\tau}$, where $\mathcal{AC}_{0,\tau}=AC([0,\tau], \M_{TV}^M(\R^d)) \times AC([0,\tau], C_{b}(\Rddiag)) $. Another application of Lemma \ref{lemma:contraction} yields the existence and uniqueness of $(\rho^2,\eta^2) \in \mathcal{AC}_{\tau,2\tau}$, the solution of \eqref{eq:ivp} on $[\tau,2\tau]$. Iterating this procedure we can construct a sequence of solutions
    \[
    (\rho^i,\eta^i) \in \mathcal{AC}_{(i-1)\tau,i\tau} \ \text{ for all } i \in \{1,\ldots,N\}, \ \qquad (\rho^{N+1},\eta^{N+1}) \in \mathcal{AC}_{N \tau, T} \ .
    \]
    We can now define the curve $(\rho,\eta) \in \mathcal{AC}_{0,T} = \ACt$ by 
    $$
\begin{cases}
(\rho_t,\eta_t)=(\rho^i,\eta^i) & \text { for all } t \in[(i-1) \tau, i \tau) \text { and } i \in\{1, \ldots, N\} \\ (\rho_t,\eta_t)=(\rho^{N+1}_t,\eta^{N+1}_t) & \text { for all } t \in[N \tau, T]\end{cases}
$$
which, by construction, is the unique solution to \eqref{eq:ivp}.
\end{proof}

\begin{remark}
    In a similar spirit of~\cite[Section 5]{esposito2022class}, we could extend the result of this section to the $L^1_\mu$ setting, i.e. $\rho \ll \mu$, which can be obtained  upon assuming $\rho_0\ll\mu$ and choosing $\lambda=\mu\otimes\mu$. In this scenario, for $\eta_t\ge0$ and symmetric for any $t\in [0,T]$ --- obtained by assuming $\eta_0(x,y)$ and $\w_t[\cdot](\cdot,x,y)$ to be symmetric and $\eta$ to satisfy the condition in Remark \ref{rem:eta_positive} --- we can prove non-negativity preservation for $\rho_t$ as in~\cite[Proposition 5.2]{esposito2022class}, i.e., if the initial condition is non-negative then the solution $\rho_t(x)\geq 0$ for a.e. $t \in [0,T]$ and $\mu$-a.e. $x \in \Rd$.
\end{remark}

\section{Fast-Slow versions of the problem} \label{sec:slow_fast_versions}
In this section we are concerned with different time scales for the evolution of the weights, motivated by~\cite[Section 2.2]{Burger_kinetic_net22}. We emphasize that the notion of solution outlined in Definition \ref{def:sol_to_ivp} is the one used throughout. Let $V:[0,T]\times \cMtv^M(\R^d)\times \Rddiag \to \V^{as}(\Rddiag)$ and $\w:[0,T]\times \M_{TV}^M(\R^d)\times \Rddiag\to \R$ satisfy~\eqref{eq:velocity_bound} and~\ref{ass:w_continuous}---\ref{ass:omega_bounded}, respectively. We start with the following version of \eqref{eq:ivp}:
\begin{equation}\label{eq:ivp_with_slow_graph}
\begin{cases}
    \partial_t \rho_t =  -\onabla \cdot F^\Phi[\mu, \eta_t ; \rho_t,V_t[\rho_t]]  
    \\
    \partial_t \eta_t = \varepsilon(\w_t[\rho_t]-\eta_t)
    \\
     \rho_0 \in \M^M_{TV}(\R^d), \ \eta_0  \in C_{b}(\Rddiag) \ ,
\end{cases}\tag{$\mathrm{Co\!-\!NCL_S}$}
\end{equation}
where $\varepsilon>0$ is a small parameter. This problem corresponds to the case in which the weights of the graph evolve at a slower time scale than the mass on the vertices. Indeed, if we formally consider that the graph evolves at a time scale $\tau = \varepsilon t$, then expressing the evolution of $\eta$ for the time scale $t$ yields \eqref{eq:ivp_with_slow_graph}.  

Theorem~\ref{thm:well-posedness} implies this problem is well-posed for $\varepsilon>0$. Then, it is natural to study the behaviour of the system~\eqref{eq:ivp_with_slow_graph} as $\varepsilon$ approaches $0^+$.  The following result shows that, as $\varepsilon \to 0^+$, the solution to~\eqref{eq:ivp_with_slow_graph} converges to the solution of a PDE on a  static graph, meaning the weight function is given by the initial condition, $\eta_0$, which does not change in time. 
\begin{theorem}\label{prop:slow_convergence}
Let $(\varepsilon_n)_{n \in \N} \subset (0,\infty)$ a vanishing sequence, that is $\varepsilon_n \to 0^+$. Consider the sequence of solutions $\{(\rho^n,\eta^n)\}_{n \in \N}$ to \eqref{eq:ivp_with_slow_graph} with $\rho^n_0 = \rho^0 \in \M^M_{TV}(\R^d)$, and $\eta^n_0 = \eta^0 \in C_{b}(\Rddiag)$ satisfying  \eqref{eq:velocity_bound}, \eqref{eq:ass_Lip_rho} and \ref{ass:w_continuous}--\ref{ass:omega_bounded}. It holds that
\[
d_\infty((\rho^n,\eta^n),(\overline{\rho},\eta^0)) \to 0, \text{ as } n \to \infty,
\]
where
\[
\overline{\rho}_t+ \int_0^t \onabla \cdot F^\Phi[\mu, \eta^0 ; \overline{\rho}_s,V_s(\overline{\rho}_s)]  \dd s = {\rho}^0 \ ,
\]
understood in duality with $C_0(\R^d)$. 
\end{theorem}
\begin{proof}
We begin noticing that Proposition \ref{prop:int_supp_mass_preservation} ensures, for any $n\in \N$, that a solution $\eta^n$ of \eqref{eq:ivp_with_slow_graph} belongs to $C([0,T],C_{b}(\Rddiag))$, satisfying the uniform bound~\eqref{eq:unif_bound_eta_T}. Thus, let $M_\eta:= \sup_{n \in \N}\norm{\eta^n}_{\infty, C_b(\Rddiag)}$. The argument follows the proof of Lemma~\ref{lemma:contraction}.

From~\eqref{eq:eta_stability} we know 
\begin{align*} 
    \norm{\eta^n_t - \eta_0}_{\infty} \leq \varepsilon_n\left( \int_0^t\norm{\eta^n_s-\eta_0}_{\infty}\dd s
 + L_\w\int_0^t\norm{\rho_s - \overline{\rho}_s}_{TV}\dd s\right)\ .
\end{align*}
Noting that $t\mapsto \int_0^t\norm{\rho_s - \overline{\rho}_s}_{TV}\dd s$ is non-decreasing, the corresponding version of Gr\"onwall's inequality implies,
\begin{align}\label{eq:estimate_for_eta_eps}
\norm{\eta^n_t - \eta_0}_{\infty}\le\varepsilon_ne^{\varepsilon_nT}(2L_\w MT)
\end{align}
Furthermore, following analogous calculations to the ones in Lemma \ref{lemma:contraction} used to obtain \eqref{eq:stability_rho}, we now compare $\rho_t^n$ and $\overline{\rho}_t$ and get  
\begin{align*}
   \norm{\rho^n_t - \overline{\rho}_t}_{{TV}} &\leq   2L_\Phi  C_V \norm{\eta^n}_{\infty, C_{b}(\Rddiag)} \int_0^t\norm{\rho^n_s - \overline{\rho}_s}_{TV}\dd s
 \\
 & \qquad + 2 L_\Phi C_V M \int_0^t\|\eta^n_s-\eta_0 \|_{ \infty}\dd s
 \\
 &\qquad +2L_\Phi  \norm{\eta_0}_{\infty} L_VM \int_0^t\norm{\rho^n_s - \overline{\rho}_s}_{TV}\dd s.
\end{align*}
Differently from Lemma~\ref{lemma:contraction}, we do not further bound the terms under the time integrals. Then, using \eqref{eq:estimate_for_eta_eps}, the uniform boundedness of $\eta^n$ and $\eta_0$ and an application of the same version of Gr\"onwall's inequality as before yields
\begin{equation}\label{eq:estimate_for_rho_eps}
    \norm{\rho^n_t - \overline{\rho}_t}_{TV} \leq\varepsilon_n e^{2 L_{\Phi}(M_\eta C_V + L_VM \|\eta_0\|_\infty) T+ \varepsilon_n T}(4 L_\Phi C_VL_\w M^2 T^2)
\end{equation}
The result follows from taking the supremum with respect to $t \in [0,T]$ in \eqref{eq:estimate_for_eta_eps} and \eqref{eq:estimate_for_rho_eps}, and letting $n$ to $\infty$.
\end{proof}

The second scenario we consider is when the graph evolves faster than the mass on the vertices, i.e.
\begin{equation} \label{eq:zero_inertia_limit}
\begin{cases}
     \partial_t \rho_t=-  \onabla \cdot F^\Phi[\mu, \eta_t ; \rho_t,V_t[\rho_t]]  
    \\
    \varepsilon \partial_t \eta_t(x,y) =\w_t[\rho](x,y)- \eta_t(x,y)  ,
    \\
     \rho_0 \in \M^M_{TV}(\R^d), \ \eta_0  \in C_{b}(\Rddiag) \ . \
\end{cases}\tag{$\mathrm{Co\!-\!NCL_F}$}
\end{equation}
System \eqref{eq:zero_inertia_limit} can be obtained by (formally) considering that the weight function $\eta$ evolves at a time scale $\tau = \frac{t}{\varepsilon}$ and then re-writing its dynamics at the time scale $t$ of the mass on the vertices. We are interested in studying the limit of this system as $\varepsilon \to 0^+$. As one can formally see from~\eqref{eq:zero_inertia_limit}, this limit would identify the weight function $\eta$ as a function depending on the mass configuration, nonlocal in case $\omega$ is as in Example~\ref{ex:omega_nonlocal}, providing a whole class of weight functions not considered so far in the literature, to the best of our knowledge.

A classical result for systems of ODEs whose orbits are contained in Euclidean space is given by the Tykhonov Theorem, cf. e.g.~\cite[Chapter 3]{banasiak2014methods}, providing suitable conditions for the $\varepsilon \to 0^+$ limit to yield the formal limit in \eqref{eq:zero_inertia_limit}. In our case, however, we do not exploit the previous result as we have an explicit solution for $\eta$ and can evaluate the limit directly.

In order to use the latter information, we require more regularity of the function $\w$, in particular, we further assume:
\begin{enumerate}[start=4,label=$(\bm{\w}\arabic*)$]
    \item $t \mapsto \w_t[\cdot](\cdot,\cdot) \in W^{1,1}([0,T])$;\label{ass:omega_sobolev_time}
    \item $ \esssup_{t \in [0,T]}\sup_{\sigma \in \cMtv(\R^d)}\sup_{x,y \in \Rddiag}\big|\partial_t\w_t[\sigma](x,y)\big| \le \widetilde{C}_\w, $\label{ass:omega_der_bounded_in_space} for some $\widetilde{C}_\w \in (0,\infty),$
\end{enumerate}
where $W^{1,1}$ stands for a Sobolev space. The additional hypotheses above are needed in order to exploit the explicit form of $\eta$ and to use integration by parts in Sobolev spaces. We note that given our assumptions \eqref{eq:velocity_bound} and \eqref{eq:ass_Lip_rho} and \ref{ass:w_continuous}-\ref{ass:omega_der_bounded_in_space}, \eqref{eq:zero_inertia_limit} is well-posed by Lemma~\ref{lemma:contraction} and Theorem \ref{thm:well-posedness} for any $\varepsilon>0$.  
\begin{example}\label{ex:omega_nonlocal}
    In this case, we can consider the following slightly modified version of the example~\eqref{eq:example_omega_intro}
\[
\w_t[\sigma](x,y) = \int_{\R^d}K(t, x,y,z)\dd\sigma(z),
\]
for $\sigma \in \M^M_{TV}(\R^d)$. We assume $K:[0,T]\times \Rddiag\times \R^d \to \R$ to satisfy: 
\begin{itemize}
    \item $t \mapsto K(t, \cdot,\cdot,\cdot) \in C^1([0,T])$.
    \item $(x,y,z) \mapsto K(\cdot, x,y,z) \in C_0(\Rddiag\times\R^d)$.
    \item $\partial_tK(t,x,y,z)\in C([0,T],C_0(\Rddiag\times\R^d))$.
\end{itemize}
The assumptions~\ref{ass:w_continuous}-\ref{ass:omega_bounded} still hold, and the derivative is given by 
 \[
 \partial_t \w_t[\sigma](x,y) =    \int_{\R^d}\partial_tK(t,x,y,z)\dd\sigma(z).
 \]
where we note that the derivative of the function $t \mapsto K*\sigma(t,\cdot,\cdot,\cdot)$ follows from an application of the Dominated Convergence Theorem. Hence, we can verify~\ref{ass:omega_sobolev_time},~\ref{ass:omega_der_bounded_in_space}.
\end{example}

\begin{remark}
Later on we actually need assumptions~\ref{ass:omega_sobolev_time}~and~\ref{ass:omega_der_bounded_in_space} for solutions of~\eqref{eq:zero_inertia_limit}, that is for time-dependent measures $\sigma$. Note that for any $(t,x,y) \in [0,T] \times \Rddiag$, we can consider the test function $z\mapsto K(\cdot,\cdot,\cdot,z)\in C_0(\Rd)$. Then, for a pair $(\rho,\eta)$ solving ~\eqref{eq:zero_inertia_limit} we know from Proposition \ref{prop:int_supp_mass_preservation} that $\rho\in\AC([0,T];\cMtv^M(\R^d))$ and, for a.e. $t\in[0,T]$, we have  
\begin{align*}
   \partial_t\w_t[\rho_t](x,y)&= \partial_t\int_{\Rd}K(t,x,y,z)\dd\rho_t(z)
   \\
   & =\int_{\Rd}\partial_t K(t,x,y,z)\dd\rho_t(z)
   \\
   &\qquad +\frac{1}{2}\iint_{\Rddiag}(K(t,x,y,v)-K(t,x,y,u))\dd F^\Phi[\mu, \eta_t ; \rho_t,V_t[\rho_t]](u,v)  \ .
\end{align*}
Therefore, assumption~\ref{ass:omega_sobolev_time} and~\ref{ass:omega_der_bounded_in_space} are satisfied by our assumptions on $K$, the flux $F^\Phi$, and the fact that $(\rho_t)_{t\in[0,T]}$ is a solution to~\eqref{eq:zero_inertia_limit}.
\end{remark}

\begin{theorem}\label{thm:Co_NCL_F}
    Let $(\varepsilon_n)_{n \in \N}\subset(0,\infty)$ such that $\varepsilon_n \to 0^+$ as $n\to \infty$, and consider a sequence of solutions $\{(\rho^n,\eta^n)\}_{n\in\mathbb{N}}$ to \eqref{eq:zero_inertia_limit} with $\eta_0^n\in C_b(\Rddiag)$ and $\rho_0^n \in \M^{M}_{TV}(\R^d)$ satisfying $\|\eta_0^n-\w_0[\rho_0^n]\|_\infty\to 0$, as $n\to\infty$. Assume $V$ and $\w$ satisfy~\eqref{eq:velocity_bound},  \eqref{eq:ass_Lip_rho}, and~\ref{ass:w_continuous}-($\bm{\w}5$), respectively. It holds 
    \[
    d_\infty((\rho^n,\eta^n), (\rho, \w)) \to 0 \text{ as } n \to \infty \ , 
    \]
    where, for a.e. $t \in [0,T]$, $\rho$ solves
    \begin{equation}\label{eq:fast_graph}
    \rho_t + \int_0^t \onabla \cdot F^\Phi[\mu, \w_s[\rho_s]; \rho_s, V_s[\rho_s]] \dd s = \rho_0 \ ,    
    \end{equation}
to be understood in duality with $C_0(\R^d)$.  
\end{theorem}
\begin{proof}
    Let us begin by considering the solution to the ODE for the weight function $\eta^n$:
\[
   \eta^n_t(x,y) = e^{-t/\varepsilon_n}\left(\eta^n_0(x,y) + \int_0^t \frac{e^{s/\varepsilon_n}}{\varepsilon_n}\w_s[\rho^n_s](x,y) \dd s\right).
\]
By our assumption~\ref{ass:omega_sobolev_time}, an integration by parts yields
\begin{align*}
  \eta^n_t(x,y) &= e^{-t/\varepsilon_n}\left(\eta^n_0(x,y) + \w_t[\rho^n_t](x,y)e^{t/\varepsilon_n} - \w_0[\rho^n_0](x,y)-  \int_0^t e^{s/\varepsilon_n}\partial_s\w_s[\rho^n_s](x,y) \dd s\right) 
  \\
  & = \eta^n_0(x,y)e^{-t/\varepsilon_n} + \w_t[\rho^n_t](x,y) - \w_0[\rho^n_0](x,y)e^{-t/\varepsilon_n}-  \int_0^t e^{-(t-s)/\varepsilon_n}\partial_s\w_s[\rho^n_s](x,y) \dd s.
\end{align*}
This implies
\begin{align*}
    \norm{\eta^n_t - \w_t[\rho_t]}_\infty & \leq \|\eta_0^n-\w_0[\rho_0^n]\|_\infty e^{-t/\varepsilon_n} + \norm{\w_t[\rho^n_t]- \w_t[\rho_t]}_\infty 
+  \int_0^t e^{-(t-s)/\varepsilon_n}\norm{\partial_s\w_s[\rho^n_s]}_\infty \dd s
    \\
    & \leq L_\w \norm{\rho_t - \rho_t^n}_{TV} + \|\eta_0^n-\w[\rho_0^n]\|_\infty +  \widetilde{C}_\w\varepsilon_n(1-e^{-t/\varepsilon_n})
\end{align*}
where in the last line we have used assumption \ref{ass:omega_der_bounded_in_space}. On the other hand, similar to the proof of Theorem~\ref{prop:slow_convergence}, we can estimate
\begin{align*}
   \norm{\rho^n_t - {\rho}_t}_{{TV}} &\leq   2L_\Phi  C_V {M_\eta} \int_0^t\norm{\rho^n_s - {\rho}_s}_{TV}\dd s
 \\
 & \qquad + 2 L_\Phi C_V M \int_0^t\|\eta^n-\w_s[\rho_s] \|_{ \infty}\dd s
 \\
 &\qquad +2L_\Phi  C_\w  L_VM \int_0^t\norm{\rho^n_s - {\rho}_s}_{TV}\dd s \ ,
\end{align*}
where we denote by $M_\eta:= \sup_{n \in \N}\norm{\eta^n}_{\infty, C_b(\Rddiag)}$.
Let 
\[
u^n(t) :=  \norm{\rho^n_t - {\rho}_t}_{{TV}}+ \norm{\eta^n_t - \w_t[\rho_t]}_\infty,
\]
and note that

\begin{align*}
    u^n(t)\leq \|\eta_0^n-\w_0[\rho_0^n]\|_\infty  + \widetilde{C}_\w\varepsilon_n(1-e^{-t/\varepsilon_n})+{C}\int_0^t u^n(s)\dd s
\end{align*}
for a suitable constant ${C}$ including all the others. By using Gr\"onwall's inequality 
\begin{equation}\label{eq:fast_gronwall}
    \begin{split}
    u^{n}(t)&\le \|\eta_0^n-\w_0[\rho_0^n]\|_\infty\!  +\! \widetilde{C}_\w\varepsilon_n(1-e^{t/\varepsilon_n})\! +\! C \int_0^t(\|\eta_0^n-\w_0[\rho_0^n]\|_\infty\! +\! \widetilde{C}_\w\varepsilon_n(1-e^{s/\varepsilon_n}))e^{C(t-s)} \dd s
    \\
    &\le\left(\|\eta_0^n-\w_0[\rho_0^n]\|_\infty  + \widetilde{C}_\w\varepsilon_n\right)\left(1+CTe^{CT}\right).
    \end{split}
\end{equation}

By taking the supremum over $t \in [0,T]$ and letting $n \to \infty$ yields the result. 
\end{proof}

\begin{remark}
As a byproduct of Theorem \ref{thm:Co_NCL_F} we obtain another existence result for solutions of~\eqref{eq:fast_graph}, i.e., equations of type~\eqref{eq:intro-CE-PDE} with a weight function depending on the mass configuration. This is, indeed, an extension of~\cite{esposito2022class} to a different class of weight functions.
\end{remark}

\section{Discrete-to-continuum limit}\label{sec:discrete_to_continuum}
  
In this section we are interested in considering problem~\eqref{eq:ivp} as the limit of a sequence of discrete finite graphs with an increasing number of vertices, i.e. a sequence of graphs whose base measure is given by a sequence of atomic measures
\begin{equation}\label{eq:atomic_measure}
\mu^n =\sum_{i=1}^n m^n_i\delta_{x_i}\ , \quad  x_i \in \R^d,\ m^n_i \in (0,\infty) \text{ for } i=1,\ldots, n  \text{ for } n\in \N.
\end{equation}  

Let us assume the uniform compressibility condition~\eqref{eq:velocity_bound} for $\mu^n$, i.e., the velocity field $V:[0,T]\times \M_{TV}^M(\R^d) \times \Rddiag \to \V^{\mathrm{as}}(\Rddiag)$ satisfies
\begin{equation}\label{eq:velocity_bound_discrete}
  \sup _{t \in[0, T]} \sup _{\rho \in \M_{TV}^M(\mathbb{R}^d)}\sup_{x\in\R^d} \sum_{\substack{j= 1 \\ x_j\neq x}}^n\left|V_t[\rho](x, x_j)\right|   \leq C_V,
\end{equation}
for a constant $C_V>0$, and the Lipschitz assumption \eqref{eq:ass_Lip_rho}, that is there is a constant $L_V \geq 0$ such that, for all $t \in [0,T]$ and all $\rho, \sigma \in \M_{TV}^M(\R^d)$,  
\begin{equation}\label{eq:ass_Lip_rho_discrete}
         \sup _{x \in \R^d} \sum_{\substack{j= 1 \\ x_j\neq x}}^n \left|V_t[\rho](x, x_j)-V_t[\sigma](x, x_j)\right|   \leq L_V\|\rho-\sigma\|_{\mathrm{TV}} \ .
\end{equation} 
Additionally, if \ref{ass:w_continuous}--\ref{ass:omega_bounded} are satisfied, the system 
\begin{equation} \label{eq:discrete_to_continuum}
\begin{cases}
     \partial_t \rho^n_t = - \onabla \cdot F^\Phi[\mu^n, \eta^n_t ; \rho^n_t,V_t[\rho^n_t]]  
    \\
    \partial_t \eta^n_t(x,y) =\w_t[\rho^n](x,y)- \eta^n_t(x,y),
    \\
     \rho_0 \in \M^M_{TV}(\R^d), \ \eta_0  \in C_{b}(\Rddiag), 
\end{cases}\tag{$\mathrm{Co\!-\!NCL_n}$}
\end{equation}
is well-posed by Theorem \ref{thm:well-posedness} for any $n \in \N$. Under suitable assumptions we will show that~\eqref{eq:ivp} can be obtained as approximation of~\eqref{eq:discrete_to_continuum}, i.e. that \eqref{eq:ivp} is a good mean field approximation for evolution problems on large finite graphs. Throughout the rest of this section we will consider the upwind interpolation for ease of presentation
\[
\Phi_{\text {upwind }}(a, b ; w)=a w_{+}-b w_{-}, \qquad \mbox{ for } (a,b,w)\in\R^3.
\]
Other admissible interpolations can be also considered, see Remark~\ref{rem:final_discrete_to_cont}. We require stronger regularity of the velocity field $V:[0,T]\times \M_{TV}^M(\R^d) \times \R^{2d} \to \V^{\mathrm{as}}(\R^{2d})$, namely, we assume the map  $\R^{2d}\ni (x,y) \mapsto V[\cdot](\cdot,x,y)\in C_0(\R^{2d})$, in view of the weak-* convergence we use below. Note that now the velocity field takes values in the space of antisymmetric vector fields over all of $\R^{2d}$, which does not allow for singularities on the diagonal $\{x=y\}$. Consequently, we also replace \eqref{eq:velocity_bound_discrete} by the slightly stronger condition
\begin{equation}\label{eq:discrete_velocity_fully_bounded}
  \sup _{t \in[0, T]} \sup _{\rho \in \M_{TV}^M\left(\mathbb{R}^d\right)} \sup _{x,y \in \R^{2d}}\left|V_t[\rho](x, y)\right|  \leq C_V. 
\end{equation}
\begin{remark}
    Note that our guiding example for the velocity field, namely $V_t[\rho_t](x,y)=-\onabla (K*\rho_t)(x,y)$
satisfies these assumptions if the kernel $K\in C_0(\R^{2d})$. 
\end{remark}
We also require more regularity of the function $\eta$ to prove the stability result with respect to $\mu$ in Theorem~\ref{thm:discrete-to-continuum} below, i.e., $\eta \in C([0,T], C_0(\R^{2d}))$. Note that it is now well defined on the diagonal and vanishes at infinity. To guarantee this requirement is met, we will need the initial datum $\eta_0 \in C_0(\R^{2d})$ and the function $\w:[0,T]\times \M_{TV}(\Rd)\times \R^{2d} \to \R$ to be such that the map $\R^{2d}\ni (x,y) \mapsto \w_t[\cdot](\cdot,x,y) \in C_0(\R^{2d})$, as well as satisfying the rest of the assumptions in \ref{ass:w_continuous}--\ref{ass:omega_bounded}. 

Below, we present a stability result with respect to the base measure. Note that in case $\mu^n$ is a sequence of atomic measures of the form \eqref{eq:atomic_measure} the next theorem is a discrete-to-continuum limit for~\eqref{eq:ivp}.
\begin{theorem}\label{thm:discrete-to-continuum}
    Fix $\Phi\equiv\Phi_{\text{upwind}}$ and consider a sequence $\{\mu_n\}_{n\in\N} \in\M_{TV}^+(\Rd)$ such that $\mu^n \overset{\ast}{\rightharpoonup} \mu \in \M^+_{TV}(\R^d)$. Let $V:[0,T]\times \M_{TV}^M(\R^d) \times \R^{2d}\to \V^{as}(\R^{2d})$ satisfy assumptions~\eqref{eq:ass_Lip_rho}~and~\eqref{eq:discrete_velocity_fully_bounded}, uniformly in $n$, and $\w:[0,T]\times \M_{TV}^M(\R^d) \times \R^{2d}\to \R^d$ satisfy~\ref{ass:w_continuous}-\ref{ass:omega_bounded}. Assume $((x,y) \mapsto V[\cdot](\cdot,x,y)) \in C_0(\R^{2d})$ and $((x,y)\mapsto \w_t[\cdot](\cdot,x,y)) \in C_0(\R^{2d})$. Let us consider a sequence of solutions $\{(\rho^n,\eta^n)\}_{n\in\N}$ to \eqref{eq:discrete_to_continuum} associated to $\{\mu_n\}$ and let $(\rho,\eta)$ be the solution to $\eqref{eq:ivp}$ depending on $\mu$. If $\norm{\rho_0^n -\rho^0}_{TV}\to 0$ and $\norm{\eta^n_0 - \eta_0}_\infty \to 0$ as $n\to \infty$, then
    \[
    \lim_{n\to \infty} d_\infty((\rho^n,\eta^n),(\rho,\eta)) =0\ .
    \]
\end{theorem}
\begin{proof}
    We begin with an estimate analogous to the one in the proof of Lemma \ref{lemma:contraction}, with an additional term that accounts for the difference in the base measure between the solutions compared, that is, for any $\varphi \in C_0(\R^d)$ we have
    \begin{align*}
        \bigg|\int_{\Rd}\varphi(x)\rho_s(x) - \int_{\Rd}\varphi(x)\rho^n_s(x)\bigg| & \le \bigg|\int_{\Rd}\varphi(x)\rho_0(x) - \int_{\Rd}\varphi(x)\rho_0^n(x)\bigg|
        \\
        & \qquad +\bigg|\int_0^t \int_{\Rddiag}\frac{\onabla \varphi(x,y)}{2}\bigg(\Phi\left(\frac{\dd\rho_s \otimes \dd \mu}{\dd \lambda}, \frac{\dd \mu \otimes \dd\rho_s}{\dd \lambda} ,V_s[\rho_s]\right)\eta_s 
        \\
        & \qquad - \Phi\left(\frac{\dd\rho^n_s \otimes \dd \mu^n}{\dd \lambda}, \frac{\dd \mu^n \otimes \dd\rho^n_s}{\dd \lambda} , V_s[\rho^n_s]\right)\eta^n_s\bigg) \dd\lambda(x,y) \dd s\bigg| \ .
        \\
        & \leq \bigg|\int_{\Rd}\varphi(x)\rho_0(x) - \int_{\Rd}\varphi(x)\rho_0^n(x)\bigg|
        \\
        &\qquad  +\bigg| \int_0^t\int_{\Rddiag}\frac{\onabla \varphi(x,y)}{2}\eta_s\bigg(\Phi\left(\frac{\dd \rho_s\otimes \dd\mu}{\dd \lambda}, \frac{\dd\mu \otimes \dd\rho_s}{\dd \lambda } ; V_s\left[\rho_s\right]\right) 
        \\
        &\qquad \qquad -\Phi\left(\frac{\dd {\rho}^n_s\otimes \dd\mu}{\dd \lambda }, \frac{\dd\mu \otimes \dd{\rho}^n_s}{\dd \lambda};V_s\left[{\rho}_s\right]\right)\bigg) \dd\lambda\dd s\bigg|
        \\
        &\qquad + \bigg| \int_0^t\int_{\Rddiag}\frac{\onabla \varphi(x,y)}{2}\eta_s\bigg(\Phi\left(\frac{\dd \rho^n_s\otimes \dd\mu}{\dd \lambda}, \frac{\dd\mu \otimes \dd\rho^n_s}{\dd \lambda } ; V_s\left[\rho_s\right]\right) 
        \\
        &\qquad \qquad -\Phi\left(\frac{\dd {\rho}^n_s\otimes \dd\mu^n}{\dd \lambda }, \frac{\dd\mu^n \otimes \dd{\rho}^n_s}{\dd \lambda};V_s\left[{\rho}_s\right]\right)\bigg) \dd\lambda\dd s\bigg|
        \\
        &\quad + \bigg| \int_0^t\int_{\Rddiag}\frac{\onabla \varphi(x,y)}{2}\left(\eta_s-{\eta}^n_s\right) 
        \\
        &  \qquad \qquad  \times \Phi\left(\frac{\dd {\rho}^n_s\otimes \dd\mu^n}{\dd \lambda}, \frac{\dd\mu^n \otimes \dd{\rho}^n_s}{\dd \lambda};V_s\left[{\rho}_s\right]\right) \dd\lambda\dd s\bigg|
        \\
        &\quad +  \bigg| \int_0^t\int_{\Rddiag}\frac{\onabla \varphi(x,y)}{2}{\eta}^n_s\bigg(\Phi\left(\frac{\dd {\rho}^n_s\otimes \dd\mu^n}{\dd \lambda}, \frac{\dd\mu^n \otimes \dd{\rho}^n_s}{\dd \lambda};V_s\left[{\rho}_s\right]\right)
        \\
        & \qquad \qquad -\Phi\left(\frac{\dd {\rho}^n_s\otimes \dd\mu^n}{\dd \lambda}, \frac{\dd\mu^n \otimes \dd{\rho}^n_s}{\dd \lambda};V_s\left[{\rho}^n_s\right]\right)\bigg)  \dd\lambda d s\bigg|
        \\
        & =:I_0 + I + II + III + IV \ . 
    \end{align*}
We can estimate the terms $I, III$ and $IV$ analogously to the proof of Lemma \ref{lemma:contraction} (see also the proof of Theorem  \ref{prop:slow_convergence}), so in the rest of this proof we focus on obtaining a bound for $II$. 
\begin{align*}
II&= \bigg|\int_0^t \int_{\R^{2d}}\frac{\onabla \varphi(x,y)}{2}\eta_s \bigg(\Phi\left(\frac{\dd\rho^n_s\otimes \dd \mu}{\dd \lambda}, \frac{\dd \mu \otimes \dd\rho^n_s}{\dd \lambda} ,V_s[\rho_s]\right)
    \\
    & \qquad \quad - \Phi\left(\frac{\dd\rho^n_s \otimes \dd \mu^n}{\dd \lambda}, \frac{\dd \mu^n \otimes \dd\rho^n_s}{\dd \lambda} , V_s[\rho_s]\right)\bigg) \dd\lambda(x,y)\dd s\bigg|
    \\
    & =\bigg| \int_0^t \int_{\mathbb{R}^{2 d}}\frac{(\varphi(y)-\varphi(x))}{2} \eta_s(x, y)\bigg(V_s\left[\rho_s\right]_+(x, y) d \rho_s^n(x) d (\mu-\mu^n)(y) 
    \\
    & \qquad \qquad  -V_s\left[\rho_s\right]_-(x, y) d (\mu-\mu^n)(x) d \rho_s^n(y)\bigg) \dd s \bigg| 
\end{align*}
\begin{align*}
& \leq \bigg| \int_0^t \int_{\mathbb{R}^{2 d}}\frac{(\varphi(y)-\varphi(x))}{2} \eta_s(x, y)V_s\left[\rho_s\right]_+(x, y) d \rho_s^n(x) d (\mu-\mu^n)(y) \dd s \bigg |
\\
& \qquad + \bigg| \int_0^t \int_{\mathbb{R}^{2 d}}\frac{(\varphi(y)-\varphi(x))}{2} \eta_s(x, y)V_s\left[\rho_s\right]_-(x, y) d(\mu-\mu^n)(x) d \rho_s^n(y) \dd s \bigg|
\end{align*}
\begin{align*}
    & \leq \sup_{t \in [0,T]} \bigg|\int_{\mathbb{R}^{2 d}}\frac{(\varphi(y)-\varphi(x))}{2} \eta_t(x, y)V_t\left[\rho_t\right]_+(x, y) d \rho_t^n(x) d (\mu-\mu^n)(y)\bigg| T
\\
&\qquad +\sup_{t \in [0,T]}  \bigg|  \int_{\mathbb{R}^{2 d}}\frac{(\varphi(y)-\varphi(x))}{2} \eta_t(x, y)V_t\left[\rho_t\right]_-(x, y) d(\mu-\mu^n)(x) d \rho_t^n(y)\bigg|T
\\
& =: II^n_1 + II^n_2.
\end{align*}

As previously mentioned, by means of analogous calculations to those in the proof of Lemma \ref{lemma:contraction} and Theorem \ref{prop:slow_convergence} we now compare $\rho_t$ and $\rho_t^n$ and obtain
\begin{align*}
     \norm{\rho_t - \rho^n_t}_{TV} \leq& \norm{\rho_0 - \rho^n_0}_{TV} + 2 L_\Phi C_V \norm{\eta}_\infty \int_0^t\|\rho_s-{\rho^n_s}\|_{TV}\dd s
 \\
 & +  2 L_\Phi C_V M \int_0^t\|\eta_s-{\eta^n_s} \|_\infty\dd s 
 \\
 & +2  L_\Phi L_VM M_\eta \int_0^t\|\rho_s-{\rho_s^n}\|_{TV}\dd s
 \\
 & + II^n_1 + II^n_2 ,
\end{align*}
where $M_\eta:= \sup_{n \in \N}\norm{\eta^n}_{\infty, C_b(\Rddiag)}$. On the other hand, using the explicit solution for the ODE describing the dynamics of $\eta^n$ and $\eta$ as well as the Lipschitz condition \ref{ass:omega_lip} yields the following estimate
\begin{align*}
    \norm{\eta_t - \eta^n_t}_\infty \leq \norm{\eta_0-\eta^n_0}_\infty + \int_0^t \norm{\eta_s - \eta^n_s}_\infty\dd s + L_\w \int_0^t \norm{\rho_s - \rho^n_s}_{TV}\dd s.
    \end{align*}
Setting
\[
u^n(t) := \norm{\rho_t - \rho^n_t}_{TV} +  \norm{\eta_t - \eta^n_t}_\infty\ , 
\]
an application of Gr\"onwall's inequality yields
\begin{align*}
   u^{n}(t) \leq \big(\norm{\rho_0 - \rho^n_0}_{TV}+\norm{\eta_0-\eta^n_0}_\infty + II^n_1 + II^n_2\big) e^{CT}\ , 
\end{align*}
for a suitable constant $C \in (0,\infty)$ depending on all the others. We now turn our attention to $II^n_1$ and $II^n_2$. Since $\{\rho_t^n\}_{n \in \N}$ is a sequence of solutions, we have that for any $t \in [0,T]$ and any $n \in \N$, $\norm{\rho^n}_{TV,\infty}\leq M$. Then, by \cite[Theorem 1.59]{ambrosio2000functions} we can extract a weakly-* converging subsequence, denoted again, with a slight abuse of notation, by $\{\rho_t^n\}_{n \in \N}$ converging to some $\bar{\rho}_t \in \M_{TV}(\Rd)$. Finally, note that the integrands in $II^n_1$  and $II^n_2$ are both in $C_0(\R^{2d})$ due to our assumptions. Together with the assumption $\mu^n \overset{\ast}{\rightharpoonup} \mu$, we have that, up to a subsequence $II^n_1$  and $II^n_2$ converge to 0 as $n \to \infty$. Having the convergence of the initial conditions by our assumptions, by taking the supremum over $t\in[0,T]$ we let $n \to \infty$, which yields the result up to passing to a subsequence. Using that the limit is the unique solution to~\eqref{eq:ivp}, we infer the convergence holds for the whole sequence.
\end{proof}
\begin{remark}\label{rem:final_discrete_to_cont}
    Theorem~\ref{thm:discrete-to-continuum} can be obtained for other admissible interpolations $\Phi$ that allow to estimate $II\le II_1^n+II_2^n$ in the proof above. For instance, we can have
    \[
        \Phi(a,b;w)=\sum_{i=1}^{M} g_i(w)(\alpha_i a+\beta_i b), \qquad \mbox{for } (a,b,w)\in\R^3,
    \]
    where $M\in\N$, $\alpha_i,\beta_i\in\R$, $g_i$ Lipschitz for any $i=1,\dots,M$. In case of the upwind interpolation $M=2$, $g_1(w)=w_+$, $\alpha_1=1$, $\beta_1=0$, $g_2(w)=w_-$, $\alpha_2=0$, $\beta_2=-1$. Another example would be the arithmetic mean, for which $M=1$, $g_1(w)=w$, $\alpha_1=\beta_1=1/2$, and similarly for other suitable mean multipliers.
\end{remark}

\subsection*{Acknowledgements}
The authors would like to thank José Antonio Carrillo and André Schlichting for fruitful discussions on the contents of the manuscript. The authors were supported by the Advanced Grant Nonlocal-CPD (Nonlocal PDEs for Complex Particle Dynamics: Phase Transitions, Patterns and Synchronization) of the European Research Council Executive Agency (ERC) under the European Union’s Horizon 2020 research and innovation programme (grant agreement No. 883363). AE was also partially supported by the EPSRC grant reference EP/T022132/1. LM is supported by the EPSRC Centre for Doctoral Training in Mathematics of Random Systems: Analysis, Modelling and Simulation (EP/S023925/1).

\bibliographystyle{abbrv}
\bibliography{biblio}

\end{document}